%% file: Newton.tex
\RequirePackage{amsmath}
\RequirePackage{amssymb}
\RequirePackage{amsfonts}

\RequirePackage{fix-cm}
\documentclass[smallextended]{svjour3}       
\smartqed  
\usepackage[pdftex]{graphicx}
\usepackage{lipsum}
\usepackage{color}
\usepackage{epstopdf} 
\usepackage{centernot}
\usepackage{algorithm,algpseudocode}
\usepackage{url}      

\ifpdf
  \DeclareGraphicsExtensions{.eps,.pdf,.png,.jpg}
\else
  \DeclareGraphicsExtensions{.eps}
\fi

\newcommand{\R}{\mathbb{R}}
\newcommand{\C}{\mathbb{C}}
\newcommand{\N}{\mathbb{N}}
\newcommand{\Z}{\mathbb{Z}}
\newcommand{\T}{\mathbb{T}}
\newcommand{\PP}{\mathbb{P}}

\newcommand{\Vc}{\mathcal{V}}

\newcommand{\Ec}{\mathcal{E}}

\newcommand{\Pc}{\mathcal{P}}

\newcommand{\Ic}{\mathcal{I}}

\newcommand{\Cheb}{\mathrm{Cheb}}
\newcommand{\Oc}{\mathcal{O}}
\newcommand{\ee}{\varepsilon}
\newcommand{\lo}{\longrightarrow}
\newcommand{\li}{\left}
\newcommand{\re}{\right}
\newcommand{\mi}{:}

\newcommand{\id}{\mathrm{id}}

\newcommand{\spann}{\mathrm{span}}
\newcommand{\nol}{||\,}
\newcommand{\nor}{\,||}

\newtheorem{experiment}{Experiment}

\begin{document}

\title{Multivariate  Newton Interpolation}



\author{Michael~Hecht \and  Karl~B.~Hoffmann \and  Bevan~L.~Cheeseman \and Ivo~F.~Sbalzarini
}


%
\date{}

\institute{MOSAIC Group, Chair of Scientific Computing for Systems Biology, Faculty of Computer Science, TU Dresden
\& Center for Systems Biology Dresden, Max Planck Institute of Molecular Cell Biology and Genetics, Dresden \at
              Pfotenhauerstr. 108
01307 Dresden
Germany\\
              \email{hecht@mpi-cbg.de, karlhoff@mpi-cbg.de, cheesema@mpi-cbg.de, ivos@mpi-cbg.de}\\
               http://mosaic.mpi-cbg.de
}


\maketitle

\begin{abstract}
For $m,n \in \N$, $m\geq 1$ and a given function 
$f : \R^m\lo \R$, the \emph{polynomial interpolation problem} (PIP) is to determine a \emph{unisolvent node set} $P_{m,n} \subseteq \R^m$ of $N(m,n):=|P_{m,n}|=\binom{m+n}{n}$ points and the uniquely defined 
polynomial $Q_{m,n,f}\in \Pi_{m,n}$ in $m$ 
variables of degree $\deg(Q_{m,n,f})\leq n \in \N$ that fits $f$ on $P_{m,n}$, i.e., $Q_{m,n,f}(p) = f(p)$, $\forall\, p \in P_{m,n}$.
For $m=1$ the solution to the PIP is well known. In higher dimensions, however, no closed framework was available. We here present a generalization of the classic Newton interpolation from one-dimensional 
to arbitrary-dimensional spaces. Further we formulate an algorithm, termed PIP-SOLVER, based on a \emph{multivariate divided difference scheme}
that computes the solution $Q_{m,n,f}$ in $\Oc\big(N(m,n)^2\big)$ time using $\Oc\big(N(m,n)\big)$ memory.
Further, we introduce unisolvent \emph{Newton-Chebyshev nodes} and show that these nodes avoid \emph{Runge's phenomenon} in the sense that arbitrary periodic
\emph{Sobolev functions} $f \in H^k(\Omega,\R) \subsetneq C^0(\Omega,\R)$, $\Omega =[-1,1]^m$ of regularity $k >m/2$ can be uniformly approximated, i.e., 
$ \lim_{n\rightarrow \infty}\nol f -Q_{m,n,f} \nor_{C^0(\Omega)}= 0$.
Numerical experiments demonstrate the computational performance and approximation accuracy of the PIP-SOLVER in practice.
We expect the presented results to be relevant for many applications, including numerical solvers, quadrature, non-linear optimization, polynomial regression, adaptive sampling, Bayesian inference, and spectral analysis.
\end{abstract}

\section{Introduction}

In scientific computing, the problem of interpolating a function $f : \R^m \lo \R$, $m \in \N$, is ubiquitous. 
Because of their simple differentiation and integration, as well as their pleasant vector space structure, polynomials $Q \in  \R[x_1,\dots,x_m]$ in $m$ variables of degree $\deg(Q) \leq n$, 
$m,n \in \N$, are a standard choice as interpolants and are fundamental in ordinary differential equation (ODE) and partial differential equation (PDE) solvers. 
For an overview, we refer to \cite{NA2} and \cite{Stoer}.  
Thus, the \emph{polynomial interpolation problem} (PIP) is one of the most fundamental problems in numerical analysis and scientific computing, formulated as:
\begin{problem}[PIP]\label{PolyInt}
Let $m,n \in \N$ and $f: \R^m\lo \R$ be a computable function. 
\begin{enumerate}
 \item[i)] Choose $N(m,n)$ nodes 
$P_{m,n}=\{p_1,\dots,p_{N(m,n)}\} \subseteq \R^m$ such that $P_{m,n}$ is unisolvent, i.e., for every $f$ there is exactly one polynomial 
$Q_{m,n,f} \in \Pi_{m,n}$ fitting $f$ on $P_{m,n}$ as $Q_{m,n,f}(p)=f(p)$ for all $p\in P_{m,n}$.
\item[ii)] Determine $Q_{m,n,f}$ once a unisolvent node set $P_{m,n}$ has been chosen.
\end{enumerate}
\end{problem}
Here, $\Pi_{m,n}$ is the vector space of polynomials $Q \in \R[x_1,\dots,x_m]$ in $m$ 
variables of degree $\deg(Q)\leq n$. Every $Q \in \Pi_{m,n}$ has $N(m,n):=\binom{m+n}{n}$ monomials/coefficients.

The function $f : \R^m \lo \R$ is assumed to be \emph{computable} in  the sense that for any $x \in \R^m$ the value of $f(x)$ can be evaluated in $\Oc(1)$ time, where $\Oc(\cdot)$ is the Bachmann-Landau symbol. Note that if $f \in \Pi_{m,n}$, then $Q_{m,n,f}=f$.
If $f : \R \lo \R$ is an arbitrary continuous function in one variable, then the \emph{Weierstrass approximation theorem} \cite{weier} guarantees that $f$ can be approximated by \emph{Bernstein polynomials}, i.e., 
there exist polynomials $Q_{n} \in \Pi_{1,n}$
such that 
$$\nol f- Q_n \nor_{C^0(\Omega)}:= \sup_{x \in \Omega}|f(x) -Q_n(x)| \xrightarrow[n\rightarrow \infty]{} 0$$ for every fixed bounded domain $\Omega\subseteq \R$.
However, these polynomials $Q_n$ are not necessarily interpolating, i.e., they need not satisfy $Q_{n}(p) = f(p)$ for all $p \in P_{1,n}$.
This additional requirement restricts the space of polynomials available to approximate $f$, which is the cause of Runge's phenomenon \cite{AT,faber,runge}: If the unisolvent nodes $P_{1,n}$ are chosen independently of $f$,
the sequence of interpolants $Q_{n,f}\in \Pi_{1,n}$, $n\in \N$, 
can diverge away from $f$, i.e., there exist at least one $f \in C^0(\Omega)$ for which
$\nol f- Q_{n,f} \nor_{C^0(\Omega)} \centernot\lo 0 $ for $n\lo \infty$.
Therefore, the approximation ability of an interpolation method depends on the choice of $P_{m,n}$ and has to be characterized. More precisely:

\begin{question}\label{Task} Let $m,n \in \N$ and $\Omega \subseteq \R^m$ be a bounded domain. Further let $S_{m,n}: C^0(\Omega,\R)\lo \Pi_{m,n}$ be an interpolation solver, i.e., 
$S_{m,n}(f)=Q_{m,n,f}$ solves the PIP for any given function 
$f : \Omega \subseteq \R^m \lo \R$, choosing
the unisolvent nodes $P_{m,n} \subseteq \Omega$ independently of $f$.
\begin{enumerate}
 \item[i)] What is the set $\mathcal{A}(S_{m,n}) \subseteq C^0(\Omega,\R)$ of continuous functions that can be approximated by $S_{m,n}$, i.e., for which 
 $\nol f -S_{m,n}(f)\nor_{C^0(\Omega)} \xrightarrow[n\rightarrow \infty]{}0 , \forall f\in\mathcal{A}$~? 
\item[ii)] How large is the absolute approximation error 
$$ \ee(f,S_{m,n}):=\nol f- S_{m,n}(f) \nor_{C^0(\Omega)} \, ? $$ 
\item[iii)] How large is the relative approximation error $\mu_{m,n} \geq 1$, such that
$$ \nol f- S_{m,n}(f) \nor_{C^0(\Omega)} \leq  \mu_{m,n}\nol f- Q^*_{m,n} \nor_{C^0(\Omega)} \quad \forall \, f \in C^0(\Omega,\R)\, $$
where $Q^*_{m,n} \in \Pi_{m,n}$ is an optimal approximation that minimizes the $C^0$-distance to $f$ ?

\end{enumerate}
\end{question}

The one-dimensional PIP (\mbox{$m=1$}) can be solved efficiently in $\Oc(N(1,n)^2)=\Oc(n^2)$ and numerically accurately by various algorithms based on \emph{Newton or Lagrange Interpolation} \cite{Stoer,berrut,gautschi,LIP}. Though, even in one dimension,
there is no efficient general method for finding an optimal node set $P_{1,m}$ that minimizes Runge's phenomenon; sub-optimal node sets can be generated efficiently. 
For $\Omega=[-1,1]$,  a classic choice of sub-optimal nodes are the roots of the \emph{Chebyshev polynomials}
\begin{equation}\label{cheby}
 \Cheb_{n}=\li\{p_k \in \R \mi p_k = \cos\li(\frac{2k-1}{2(n+1)}\pi\re)\,, k=1,\dots,n+1\re\}\,, 
\end{equation}
which are optimal up to a factor depending on the \mbox{$(n+1)$}-th derivative of $f$. Therefore, the approximation ability of \emph{Chebyshev nodes} is characterized by: 
\setcounter{theorem}{-1}
\begin{theorem}\label{Cheb} Let $S_n: C^0(\Omega,\R)\lo \Pi_n$, $\Omega=[-1,1]$, be an interpolation solver that uses $\Cheb_{n}$ as unisolvent nodes, i.e., $S_n(f)=Q_{n,f}$, $Q_{n,f}(\Cheb_{n})=f(\Cheb_{n})$.
  \begin{enumerate}
 \item[i)] The set of approximable functions satisfies $C^1(\Omega,\R) \subsetneq \mathcal{A}(\Omega,\R) \subsetneq C^0(\Omega,\R)$.
 \item[ii)] If  $f \in C^{n+1}(\Omega,\R)$  and $x \in \Omega$ then  the absolute approximation error at $x$ can be bounded by 
 $$| f- S_{n}(f)(x) | \leq \frac{1}{(n+1)!}f^{(n+1)}(\xi_x) \prod_{i=1}^{n+1} (x-p_i) \leq \frac{f^{(n+1)}(\xi)}{2^n(n+1)!}\,, \xi_x \in \Omega\,, $$
 \item[iii)] The relative approximation error can be bounded by the \emph{Lebesgue function} $\Lambda(\Cheb_n)$ as: 
 $$ \nol f- S_{m,n}(f) \nor_{C^0(\Omega)} \leq  (1+\Lambda(\Cheb_n) )\nol f- Q^*_{n} \nor_{C^0(\Omega)} \,,$$
 where $\Lambda(\Cheb_n)  = \frac{2}{\pi}\big(\log(n) + \gamma +  \log(8/\pi)\big) + \Oc(1/n^2) $
 and $\gamma \sim 0.5772$ is Euler's constant \cite{brutman}.
\end{enumerate}
\end{theorem}
This provides a pleasing solution to the PIP in one dimension. We refer to \cite{AT,gautschi,brutman,burden,Stewart} for further details and proofs.

However, many data sets in scientific computing are functions of more than one variable and therefore require multivariate polynomial interpolation. 
A solution to the multivariate PIP, complete with a computationally efficient and numerically stable algorithm for computing it, has so far not been available.
This is at least partly due to the fact that an efficiently computable characterization of unisolvent nodes in arbitrary dimensions \mbox{$m \in \N$} was not known. 
In one dimension, unisolvent node sets are characterized by the simple requirement that nodes have to be pairwise different, which can obviously be asserted in $\Oc(n^2)$ time. 
While some unisolvent node sets have been proposed in dimensions \mbox{$m=2,3$} \cite{Bos,Erb,Gasca2000,Chung}, generalizations to arbitrary dimensions had a complexity that prohibited their practical 
implementation \cite{FAST,Gasca2000}. Available PIP solvers in higher dimensions therefore use randomly generated node sets and then determine the interpolation polynomial by numerically 
inverting the resulting \emph{multivariate Vandermonde matrix} $V_{m,n} \in \R^{N(m,n)\times N(m,n)}$ in order to compute the coefficients $C_{m,n}$ of $Q_{m,n,f}$ in normal form. 
Using random node sets is possible due to the famous theorem of Sard, which was later generalized by Smale \cite{smale}. This theorem states that the superset $\Pc_{m,n}$ of all unisolvent node sets 
for $m,n \in \N$ is a set of second category in the sense of Baire. Therefore, any randomly generated node set is unisolvent with probability 1.

Using random nodes, however, can never guarantee numerical stability of the solver, nor can it control Runge's phenomenon. In addition, numerical inversion of the multivariate Vandermonde matrix $V_{m,n}$ in practice incurs a computational cost larger than that of Newton interpolation. In principle, inverting the Vandermonde matrix should be as complex as solving the PIP, due to the special structure of $V_{m,n}$. However, this structure depends on the choice of unisolvent nodes $P_{m,n}$. Therefore, using random node sets prevents exploiting this structure, so that solving the system of  linear equations 
$$ V_{m,n}(P_{m,n})C_{m,n}=F\,, \quad  F= (f(p_1),\dots,f(p_{N(m,n)}))^T \in \R^{N(m,n)}$$ 
still requires the same computational time as \emph{general matrix inversion}. A lower complexity bound for inverting general matrices of size $N \times N\,, N\in \N$, is given by $\Oc(N^2\log(N))$ \cite{cormen,raz,tveit}. 
The fastest known algorithm for general matrix inversion is the \emph{Coppersmith-Winograd algorithm} \cite{COPPER}, 
which requires runtime in $\Oc(N^{2.3728639})$ in its most efficient version \cite{FAST}.  However, the \emph{Coppersmith-Winograd algorithm} is rarely used in practice, because it is only advantageous for matrices so large that memory problems become prevalent on modern hardware \cite{robinson}.
The algorithm that is mostly used in practice is the \emph{Strassen algorithm} \cite{strassen}, which runs in $\Oc(N^{2.807355})$. 
Alternatively,  one can perform \emph{Gaussian elimination} in $\Oc(N^3)$. All of these approaches require $\Oc(N^2)$ memory to store the matrix. 
Moreover, the numerical robustness and accuracy of these approaches is limited by the condition number of the Vandermonde matrix, which 
again depends on the choice of $P_{m,n}$ and can therefore  not be controlled when using random node sets. 
Hence, previous approaches to polynomial interpolation become inaccurate or intractable with increasing $N(m,n)$.

\subsection{Statement of Contribution}
Though the relevance of multivariate interpolation is undisputed and feasible interpolation schemes in dimension 1 are known since the 18$^{th}$ century,
there was so far no general interpolation scheme for multivariate functions that can guarantee to solve the PIP numerically 
robustly and accurately, controls Runge's phenomenon, and is as computationally efficient as Newton interpolation. 

We here close this gap  by introducing the notion of \emph{multivariate Newton polynomials} and a characterization of \emph{unisolvent nodes} $P_{m,n}$ such that all of the following is true:
\begin{enumerate} 
 \item[R1)] The unisolvent nodes $P_{m,n}$ generate a well-conditioned Vandermonde matrix such that the interpolant can be computed numerically robustly and accurately.
 \item[R2)] The unisolvent nodes $P_{m,n}$ generate a \emph{lower triangular} Vandermonde matrix with respect to the multivariate Newton polynomials, 
            such that the interpolant can be computed in quadratic time using a \emph{multivariate divided difference scheme}.
 \item[R3)] The unisolvent nodes $P_{m,n}$ control Runge's phenomenon in the sense that Theorem \ref{Cheb} generalizes to dimension $m$ and thereby answers Question \ref{Task}.
 \end{enumerate}
 
We practically implement the solution as an algorithm, called PIP-SOLVER, which, allows to interpolate arbitrary multivariate functions $ f : \R^m \lo \R$ numerically accurately. 
The PIP-SOLVER algorithm is based an a recursive decomposition approach, which yields a \emph{recursive generator} of \emph{unisolvent node sets} 
and the associated \emph{multivariate divided difference scheme} to determine the interpolant $Q_{m,n,f}$. We show that the PIP-SOLVER requires runtime in $\Oc(N(m,n)^2)$ and memory in $\Oc(N(m,n))$,
which matches the performance of Newton interpolation for \mbox{$m=1$}.
Further, the multivariate Newton--form of $Q_{m,n,f}$ allows  evaluating and differentiating $Q_{m,n,f}$  in linear time. 

Moreover, we show that all Sobolev functions $f \in H^k(\Omega,\R) \subseteq C^0(\Omega,\R)$, $\Omega =[-1,1]^m$ for $k >m/2$ that are periodic on $\Omega$
can be approximated uniformly 
$$\nol f- Q_{m,n,f}\nor_{C^0(\Omega)} \xrightarrow[n\rightarrow \infty]{}0 \quad \text{for $\Omega=[-1,1]^m$}$$
by using unisolvent nodes of \emph{generalized Newton-Chebyshev type}. 
Analogous to the one-dimensional case, we provide bounds for the relative and absolute approximation errors. 
Note that this is probably the best result one can obtain, since uniform approximation of non-continuous functions by polynomials and extrapolating functions $f : \Omega' \supseteq \Omega \lo \R$  with $P_{m,n}\subseteq \Omega$ 
are ill-posed problems.
Because $H^k(\Omega,\R)$ contains non-continuous functions for $k \leq m/2$, the Sobolev functions $ f \in H^k(\Omega,\R)$ for $k>m/2$ are the largest \emph{Hilbert space} of continuous functions
that can be approximated by polynomials in the $C^0$--sense. Further, the Sobolev space, $H^k(\Omega,\R)$, $k >m/2$, densely contains all \emph{smooth} and thus all \emph{analytical functions}. 
The periodic boundary condition can always be achieved by rescaling. For example, consider $\widetilde f $ periodic on $\Omega_2=[-2,2]^m$ with $\widetilde f_{|\Omega }=f$ 
and rescale $\Omega_2$ to $\Omega$. 
For these and other reasons, $H^k(\Omega,\R)$, $k >m/2$, is the pivotal analytical choice in scientific computing \cite{Jost}.

\subsection{Paper Outline}  After stating the main results of this article in section \ref{main}, we recapitulate classic one-dimensional Newton interpolation in section \ref{Newt} 
and previous multivariate interpolation schemes in section \ref{FormVander}. In section \ref{Decomp} we provide the mathematical proofs for our results. In section \ref{RPIP}, we present the
\emph{multivariate divided difference scheme} for generalized \emph{multidimensional Newton nodes} as an efficient solver algorithm. In section \ref{APT}, we show that multivariate Newton-Chebyshev nodes 
allow proving the uniform approximation result for periodic Sobolev functions $f \in H^k(\Omega,\R)$, $k>m/2$, and the bounds on the approximation errors. 
The results are then demonstrated in numerical experiments in section \ref{EX}.
Finally, we sketch a few of the possible applications in section \ref{APP} and conclude in section \ref{Conc}.

\section{Main Results}\label{main}
We summarize the main results of this article in the following three Theorems. They are based on the realization that the PIP can be decomposed into sub-problems of lesser dimension or lesser degree. 
Recursion then decomposes the problem along a binary tree whose leafs are associated with constants or zero-dimensional sub-problems, 
which only require evaluating $f$  to be solved. 
The decomposition is based on the notion of multivariate Newton polynomials, which we introduce in section \ref{RPIP}, generalizing Newton nodes to arbitrary dimensions and defining a multivariate Newton basis.

 \begin{theorem}[Main Result I] \label{I} Let $m,n \in \N$ and $f : \R^m \lo \R$ be a given function. Then, there exists an algorithm with runtime complexity $\Oc\big(N(\mbox{m,n})^2\big)$ requiring
 $\Oc(N(\mbox{m,n}))$ memory that computes: 
 \begin{enumerate}
  \item[i)] A unisolvent node set $P_{m,n} \subseteq \R^m$ and the coefficients of the corresponding interpolation polynomial $Q_{m,n,f} \in \Pi_{m,n}$ in normal form.
  \item[ii)] A unisolvent set $P_{m,n} \subseteq \R^m$ of multidimensional Newton nodes and the coefficients of the corresponding interpolation polynomial $Q_{m,n,f} \in \Pi_{m,n}$ in multivariate Newton form.
 \end{enumerate}
  \label{M1}
 \end{theorem}

We generalize classic results for Newton polynomials to show that the present multivariate Newton form is a good choice for multivariate polynomial interpolation. 
In particular, it enables efficient and accurate numerical computations according to the second main result:
  \begin{theorem}[Main Result II] \label{II} Let $m,n \in \N$ and a set of multidimensional Newton nodes $P_{m,n}$ be given. Let further $Q \in \Pi_{m,n}$ 
  be a polynomial given in multivariate Newton form. Then, there exist algorithms  that compute: 
  \begin{enumerate}
   \item[i)] The value of $Q(x_0)$  for any $x_0 \in \R^m$  in 
  $\Oc(N(m,n))$;   
   \item[ii)] The partial derivative  $\partial_{x_i} Q|_{x_0}$ for any $i \in \{1,\dots,m\}$ and $x_0 \in \R^m$ in 
  $\Oc(nN(m,n))$; 
     \item[iii)]  The integral $\int_{\Omega}Q(x)\mathrm{d}x$ for any hypercube $\Omega \subseteq \R^m$ with runtime complexity in 
  $\Oc(nN(m,n))$.
   \label{M2}
  \end{enumerate}
  \end{theorem}
  
Studying the approximation properties of multivariate Newton interpolation, we show that periodic Sobolev functions $f \in H^k(\Omega,\R)$ for $k >m/2$ can be approximated uniformly, 
and we provide upper bounds for the corresponding absolute and relative approximation errors. To do so let $A_{m,n}\subseteq \N^m$ be the set of all multi--indices $\alpha=(\alpha_1,\dots,\alpha_m) \in \N^m$ with $|\alpha|=\sum_{i=1}^m\alpha_i \leq n$ .
 \begin{theorem}[Main Result III] \label{III} Let $m,n \in \N$, $\Omega= [-1,1]^m \subseteq \R^m$. 
 Let $S_{m,n}:H^k(\Omega,\R) \lo \Pi_{m,n}$, $k>m/2$, denote an interpolation operator with respect to interpolation nodes $P_{m,n}$.
    \begin{enumerate}
  \item[i)] If $P_{m,n}$ are of Newton type, then the Lebesgue function
\begin{align*}
  \Lambda(P_{m,n}, H^k(\Omega)) &= \sup_{f\in H^k(\Omega,\R), ||f||_{H^k(\Omega)} =1} ||S_{m,n}(f)||_{C^0(\Omega)}\,, 
\end{align*}
given by the operator norm of $S_{m,n}$ is bounded by 
  \begin{align*}
   \Lambda(P_{m,n}, H^k(\Omega)) &\leq  \prod_{i=1}^m\Lambda(P_i,H^{k-(m-1)/2}(\Omega))\,.
  \end{align*}
  \item[ii)] If $P_{m,n}$ are of Newton-Chebyshev type, then  $\Lambda(P_{m,n}, H^k(\Omega))  \in \Oc(\log(n)^m)$ 
  and every $f \in H^k(\Omega,\R)$, $k >m/2$, being periodic on $\Omega =[-1,1]^m$ can be approximated uniformly:
  $$\nol f -S_{m,n}(f)\nor_{C^0(\Omega)} \xrightarrow[n\rightarrow \infty]{}0 \,.$$
  \item[iii)] If $f \in C^{n+1}(\Omega,\R)$ and $P_{m,n}$ are of Newton--type then for every $\alpha \in A_{m,n}\setminus A_{m,n-1}$, $i \in\{1,\dots,m\}$ and every $x \in \Omega$ there is 
  $\xi_x\in \Omega$ such that 
\begin{equation*}
   | f(x)- S_{m,n}(f)(x) |\leq \frac{1}{\alpha_i!} \partial^{\alpha_i+1}_{x_i}f(\xi) |N_\alpha(x)|\,, 
\end{equation*}
where  $N_\alpha(x) = \prod_{i=1}^m\prod_{j=1}^{\alpha_i}(x_i-p_{i,j})$, $x =(x_1,\dots,x_m)$, $p_{i,j}\in \PP_i$.
If  $P_{m,n}$ are of Newton-Chebyshev--type then we can further estimate 
\begin{equation*}
| f(x)- S_{m,n}(f)(x) | \leq \frac{1}{2^{\alpha_i}\alpha_i! }\partial^{\alpha_i+1}_{x_i}f(\xi_x)\,.
\end{equation*}
  \item[iv)] 
 For any unisolvent node set $P_{m,n}$  and every $f \in H^k(\Omega,\R)$ the relative error is bounded by 
  $$\nol f- S_{m,n}(f) \nor_{C^0(\Omega)} \leq  \big(1 + \Lambda(P_{m,n}, H^k(\Omega))\big)\nol f- Q^*_{m,n} \nor_{C^0(\Omega)}\,, $$
  where $Q^*_{m,n}$ is an optimal approximation that minimizes the $C^0$-distance to $f$.
 \end{enumerate}
  \end{theorem}
$\Lambda$ is estimated in $(i)$, but statement $(iv)$ holds for all unisolvent node sets. For Chebyshev nodes, $(iv)$ provides a tight estimate.

\section{Newton Interpolation in Dimension 1}\label{Newt}
Since our $m$-dimensional generalization is a natural extension of \emph{Newton interpolation} in 1D, we first review some classic results in the special case of dimension \mbox{$m=1$}, 
with more detailed discussions available elsewhere \cite{atkinson,endre,LIP,powel,NA2,Stewart,Stoer,walston}. 
In one dimension, the Vandermonde matrix $V_{n}(P_n)$ takes the classical form 
$$ V_{n}(P_{n}) = \li(\begin{array}{cccc}
                       1   & p_1    & \cdots      & p_1^{ n}  \\
                       \vdots & \vdots & \ddots& \vdots \\
                       1 & p_{n+1}   & \cdots & p_{n+1}^{n}\\
                      \end{array}\re)
                      \,.$$
It is well known that for this matrix to be regular, the nodes $p_1,\dots,p_{n+1}$ have to be pairwise distinct. As we see later, this is also a sufficient condition for the nodes $P_{n}$ to be unisolvent.
In light of this fact, we observe that $V_{n}$ induces a vector space isomorphism  
$\varphi : \R^{n+1} \lo \Pi_{n}$, where $\varphi(v)$ is the polynomial with normal-form coefficients $C_n \in \R^{n+1}$ such that $V_{n}C_n=v$, $v \in \R^{n+1}$. 
The polynomials
\begin{equation}\label{NewtPoly}
 N_i(x) = \prod_{j=1}^{i} (x-p_j) \,, \quad i =0,\dots,n
\end{equation}
are the \emph{Newton Basis} (NB) of $\Pi_{n}$.
When represented with respect to the NB, the Vandermonde matrix becomes a lower triangular matrix of the form
$$ V_{\mathrm{NB},n}(P_{n}) = \li(\begin{array}{cccc}
                       1   & 0    & \cdots      &0 \\
                      1  & (p_2-p_1) & \cdots &0\\
                        1 & (p_3 -p_1) & (p_3-p_1)(p_3-p_2)& \vdots \\
                        \vdots & \vdots & \ddots& \vdots \\
                       1 & (p_{n+1} -p_1)  & \cdots & \prod_{j=1}^n(p_{n+1}-p_j)\\
                      \end{array}\re)
                      \,.$$
Thus, the solution of the PIP in dimension $m=1$ with respect to the NB can directly be obtained as:
\begin{equation}
 Q_{n,f}(x) = \sum_{i=0}^{n}c_iN_i(x)= \sum_{i=0}^{n}c_i \prod_{j=1}^{i} (x-p_j) \,.
\end{equation}
More efficiently, the \emph{Aitken-Neville} or \emph{divided difference} scheme determines the coefficients $c_i$ by setting 
$$[p_1]f:=f(p_1)\,,\quad \quad [p_i,\dots,p_j]f:= \frac{[p_i,\dots,p_{j-1}]f-[p_{i+1},\dots,p_j]f}{x_j-x_i}\,, \,\, j\geq i $$
and proving that $c_{i-1}=[p_1,\dots,p_{i}]f$. Indeed one can verify by induction that $Q_{n,f}$ satisfies the fitting condition $f(P_n)=Q_{n,f}(P_n)$, which uniquely determines $Q_{n,f}$ up to its representation.
The induced recursion can be illustrated as follows: 
\begin{equation}\label{DDS}
 \begin{array}{crcrccrcrc}
[p_1]f \\
       & \searrow \\{}
[p_2]f & \rightarrow  & [p_1,p_2]f  \\
       & \searrow     &                & \searrow     \\{}
[p_3]f & \rightarrow  & [p_2,p_3]f     & \rightarrow & [p_1,p_2,p_3]f \\{}
 \vdots & \vdots      & \vdots         & \vdots    & \vdots  &\ddots \\{}
  & \searrow     &                & \searrow    & &              & \searrow \\{}
[p_{n+1}]f & \rightarrow  & [p_{n},p_{n+1}]f & \rightarrow & [p_{n-1},p_{n},p_{n+1}]f
  & \cdots & \rightarrow & [p_1\ldots p_{n+1}]f .  \\
\end{array}
\end{equation}

We summarize some facts about the classic 1D scheme \cite{gautschi,Stoer}, as they are important for the generalization to arbitrary dimensions: 
\begin{proposition}\label{bN} Let $n \in \N$ and $f \in C^0(\Omega,\R)\,, \Omega=[-1,1]$ 
 \begin{enumerate}
 \item[i)] The divided difference scheme allows to numerically stable determine the $\mbox{n+1}$ coefficients $c_0,\dots,c_n \in \R$ of the interpolant $Q_{n,f}$ with respect to the Newton--basis in $\Oc(n^2)$.
  \item[ii)]  Given the interpolant $Q_{n,f}$ in its Newton--form, the  \emph{Horner-scheme} allows to compute the value of $Q_{n,f}(x)$ in $\Oc(n)$ for any $x \in \Omega$. 
  \item[iii)] Given the interpolant $Q_{n,f}$ in its Newton--form, the  value of the derivative $\frac{d}{dx}Q_{n,f}(x)$  can be computed in $\Oc(n)$ for any $x \in \Omega$. 
 \end{enumerate}

\end{proposition}

Due to its relationship with \emph{Taylor expansion}, \emph{Newton interpolation} has several advantages over other interpolation schemes.
For instance, it easily extends to higher degrees without recomputing the coefficients of lower-order terms, i.e., by incrementally computing higher-order terms. 
Furthermore, evaluation of the Newton interpolant and its derivatives is straightforward due to their simple form. 
However, the recursive nature of the divided difference scheme also has certain disadvantages compared to \emph{Lagrange interpolation}, which computes the result at once. 
We refer to \cite{berrut} and \cite{werner} for further discussions of the properties of these schemes.

In any case, due to the uniqueness of the interpolation polynomial, the question how well the interpolant approximates a given function is independent of the specific interpolation scheme, up to numerical rounding errors. 
Instead, the approximation quality only depends on the choice of interpolation nodes. The approximation quality in the 1D case is given in Theorem \eqref{Cheb}.

\section{Multivariate Polynomials} \label{FormVander}

We follow the notation of P.J. Olver \cite{MultiVander} to extend our considerations to the general multivariate case.
For $m,n\in \N$, we denote by $\Pi_{m,n}\subseteq \R[x_1,\dots,x_m]$ the vector space of all real polynomials in $m$ variables of bounded degree $n$. 
While the normal form of a polynomial \mbox{$Q \in \Pi_{m,n}$} possesses $N(m,n):=\binom{m+n}{n}$ 
monomials, the number of monomials of degree $k$ is given by  $M(m,k):=\binom{m+k}{m} - \binom{m+k-1}{m}$.
We  enumerate the coefficients $c_0,\dots,c_{N(m,n)}$ of $Q \in \Pi_{m,n}$ in its normal form as follows:
\begin{align}
 Q(x) &=  c_0 + c_1 x_1  + \cdots + c_mx_{m}+ c_{m+1}x_1^2+c_{m+2}x_1x_2  \nonumber \\ 
 & + \cdots  + c_{2m}x_1x_m + c_{2m+1}x_2^2+  \cdots  +  c_{M(m,n-1)+1}x_1^n   \nonumber \\ 
 & +c_{M(m,n-1)+2}x_1^{n-1}x_2+ \cdots + c_{N(m,n)-2}x_{m-1}x_m^{n-1}\nonumber \\
 & + c_{N(m,n)-1}x_m^n\,. \label{C}
\end{align}
We assume that \mbox{$0 \in \N$} and 
consider $A_{m,n}:=\li\{\alpha \in \N^m : |\alpha| \leq n\re\}$ 
the set of all multi-indices of order  $|\alpha|:= \sum_{k=1}^m\alpha_k \leq n$.
 The multi-index is used to address the monomials of a multivariate polynomial. 
For a vector $x=(x_1,\dots,x_m)$ and $\alpha \in A_{m,n}$ we define 
\begin{equation*}
x^\alpha:=  x_1^{\alpha_1}\cdots x_m^{\alpha_m}
\end{equation*}
and we order the $M(m,k)$ multi-indices of order $k$ with respect to lexicographical order, i.e., 
$\alpha_1 =(k,0,\dots,0)$, $\alpha_2=(k-1,1,0\dots,0)$,$\dots$, $\alpha_{M(m,k)} =(0,\dots,0,k)$. 
The $k$-th \emph{symmetric power} 
$x^{\odot k }=(x_1^{\odot k },\dots,x_{M(m,k)}^{\odot k }) \in \R^{M(m,k)}$ is defined  by the entries 
\begin{equation}
x^{\odot k }_i:= x^{\alpha_i} \,, \quad i=1,\dots,M(m,k)\,.
\end{equation}
Thus, $x^{\odot 0 }_i =1$, $x^{\odot 1 }_i =x_i$.  
\begin{definition}[multivariate Vandermonde matrix] For given \mbox{$n,m \in \N$} and a set of nodes $P=\{p_1,\dots,p_{N(m,n)}\} \subseteq \R^m$ 
with $p_i = (p_{1,i}, \dots,p_{m,i})$, 
we define the \emph{multivariate Vandermonde matrix} $V_{m,n}(P)$ by 
$$ V_{m,n}(P) = \li(\begin{array}{ccccc}
                       1   & p_1   & p_1^{\odot 2}  & \cdots      & p_1^{\odot n}  \\
                       1  & p_2 &p_2^{\odot 2} & \cdots & p_2^{\odot n}\\
                       1 & p_3 & p_3^{\odot 2}& \cdots & p_3^{\odot n}\\
                       1 & \vdots &\vdots & \ddots& \vdots \\
                       1 & p_{N(m,n)}  & p_{N(m,n)}^{\odot 2} & \cdots & p_{N(m,n)}^{\odot n}\\
                      \end{array}\re)
                      \,.$$ 
                      \end{definition}

We call a set of nodes $P =\{p_1,\dots,p_{N(m,n)}\} \subseteq \R^m$ \emph{unisolvent} if and only if the  Vandermonde matrix $V_{m,n}(P)$
is regular. Thus, the set of all unisolvent node sets is given by 
$\Pc_{m,n}=\li\{ P \subseteq \R^m : \det\big(V_{m,n}(P)\big)\not = 0\re\}$, which is an open set in $\R^m$, since $P \mapsto \det\big(V_{m,n}(P)\big)$ is a continuous function. 

Given a real-valued function $f :R^m \lo \R$, and assuming that unisolvent nodes $P=\{p_1,\dots p_{N(m,n)}\} \subseteq \R^m$ exist, the linear system of equations $V_{m,n}(P)x = F$, with
$$  x=\big(x_1,\dots,x_{N(m,n)}\big)^\mathsf{T}\,\,\,\text{and}\,\,\, F=\big(f(p_1),\dots,f(p_{N(m,n)})\big)^\mathsf{T} , $$ 
has the unique solution
$$ x = V_{m,n}(P)^{-1}F \,.$$ 
Thus, by using $c_i:=x_i$ as the coefficients of $Q\in\Pi_{m,n}$, enumerated as in Eq.~\eqref{C}, we have uniquely determined the solution of Problem \ref{PolyInt}. 
The essential difficulty herein lies in finding a good unisolvent node set posing a well-conditioned problem and in solving $V_{m,n}(P)x =F$ accurately and efficiently.  
Therefore, we state a well-known result, first mentioned in \cite{Chung} and then again in \cite{MultiVander}, saying:

\begin{proposition}\label{generic} Let $m,n \in \N$  and $P_{m,n} \subseteq \R^m$, $\# P_{m,n} = N(m,n)$. The Vandermonde matrix $V_{m,n}(P_{m,n})$
is regular if and only if 
the  nodes $P_{m,n}$  
do not belong to a common algebraic hypersurface of
degree $\leq n$, i.e., if there exists no polynomial $Q \in \Pi_{m,n}$, such that $ Q(p) = 0 $ for all $p \in P_{m,n}$.  
\end{proposition}

\begin{proof} Indeed $P_{m,n}$ is unisolvent if and only if
the homogeneous Vandermonde problem 
$$V_{m,n}(P)x = 0$$
has no non-trivial solution, which is equivalent to the fact that there is no polynomial $Q \in \Pi_{m,n}\setminus\{0\}$ generating a hypersurface $W =Q^{-1}(0)$ of degree $\deg(W)\leq n$ with $P_{m,n} \subseteq W$.  
\qed\end{proof}

The geometric interpretation of Proposition \ref{generic} is crucial but not constructive, i.e., it classifies node sets but yields no algorithm or scheme to construct unisolvent node sets. 
In the next section, we therefore further develop our understanding of unisolvent node sets in order to provide a construction algorithm, which is our first main result.

\section{Multivariate Interpolation}\label{Decomp} 
We provide theorems that allow solving the PIP in a way that fulfills requirements $(R1)$ and $(R2)$ stated in the introduction. We start by considering the multivariate interpolation problem on hyperplanes and then show that the general multivariate PIP can be decomposed into such hyperplane problems.

\subsection{Interpolation on Hyperplanes}\label{SC}
We consider the PIP on certain subplanes and therefore define:
\begin{definition}[affine transformation] 
Let $\tau : \R^m \lo \R^m$ be given by $\tau(x) = Ax +b $, where $A \in \R^{m\times m}$ is a full-rank matrix, i.e. $\mathrm{rank}(A) =m$, and $b \in \R^m$. Then we call 
$\tau$ an \emph{affine transformation} on $\R^m$. 
\end{definition}
\begin{definition}\label{proj}
For every ordered tuple of integers $i_1,\dots,i_k \in \N$, $i_{q} < i_p$ if $1 \leq q <p \leq k$, we consider  
$$H_{i_1,\dots, i_k}=\li\{(x_1,\dots,x_n) \in \R^m \mi x_j = 0 \,\, \text{if}\,\, j \not \in \{i_1,\dots,i_k\}\re\}$$ the $k$-dimensional hyperplanes spanned by the 
$i_1,\dots,i_k$-th coordinates. We denote by $ \pi_{i_1,\dots, i_k} : \R^m \lo H_{i_1,\dots,i_k}$ and $i_{i_1,\dots, i_k} : \R^k \hookrightarrow  \R^m$, with $i_{i_1,\dots, i_k}(\R^k)=H_{i_1,\dots,i_k}$, 
the natural projections and embeddings. 
We denote by
$\pi_{i_1,\dots,i_k}^* : \R[x_1,\dots,x_m] \lo \R[x_{i_1},\dots x_{i_k}]$ and $i^*_{i_1,\dots, i_k} : \R[y_{1},\dots y_{k}] \hookrightarrow \R[x_1,\dots,x_m]$ the induced projections and embeddings 
on the polynomial ring. 
\end{definition}
\begin{definition}[solution of the PIP]\label{Hxi} Let $m,n \in \mathbb{N}$, 
$\xi_1,\dots,\xi_m \in \R^m$ an orthonormal frame (i.e., $\li<\xi_i,\xi_j\re>=\delta_{ij}$, $ \forall 1\leq i,j\leq m$, where $\delta_{ij}$ denotes the Kronecker symbol), and $b \in \R^m$. 
For $\Ic \subseteq \{1,\dots,m\}$
we consider the hyperplane 
$$H_{\Ic,\xi,b}:= \li\{x \in \R^m \mi \li<x-b,\xi_i\re>=0\,, \forall \, i \in \Ic \re\}\,.$$ 
Given a function $f : \R^m \lo \R$ we say that a set of nodes $P_{k,n} \subseteq H_{\Ic,\xi,b}$ and a polynomial $Q \in \Pi_{m,n}$
\emph{solve the PIP} with respect to $k=\dim H_\Ic$, $n\in \N$ and $f$ on $H_\Ic$ if and only if $Q(p) =f(p)$ for all $p \in P_{k,n}$ and whenever there is  a $Q'\in \Pi_{m,n}$ and $Q'(p)=f(p)$
for all $p \in P_{k,n}$, then $Q'(x)=Q(x)$ for all $x \in H_{\Ic,\xi,b}$.
\end{definition}
\begin{definition}[induced transformation]Let  $H \subseteq \R^m$ be a hyperplane of dimension $k\in \N$ and $\tau : \R^m\lo \R^m$ an affine transformation such that 
$\tau(H) =H_{1,\dots,k}$.
Then we denote by
$$\tau^*: \R[x_1,\dots,x_m] \lo \R[x_1,\dots,x_m]$$
the \emph{induced transformation} on the polynomial ring defined over the monomials as:
$$ \tau^*(x_i) = \eta_1x_1+\cdots + \eta_mx_{m}\quad \text{with} \quad \eta=(\eta_1,\dots,\eta_m) = \tau(e_i)\,,$$
where $e_1,\dots,e_m$ denotes the standard Cartesian basis of $\R^m$.
\end{definition}
\begin{lemma}\label{TT} Let $m,n \in \N$ and $P_{m,n}$ be a unisolvent node set with respect to $(m,n)$. Further let $\tau : \R^m \lo \R^m$, $\tau(x) =Ax +b$ be an affine transformation. 
Then, $\tau(P_{m,n})$ is also a unisolvent node set with respect to $(m,n)$.
\end{lemma}
\begin{proof} Assume there exists a polynomial \mbox{$Q_0 \in \Pi_{m,n}\subseteq \R[x_1,\dots,x_m]$} such that \mbox{$Q_0(\tau(P_{m,n}))=0$}. Then setting \mbox{$Q_1:=\tau^*(Q_0)$} yields a non--zero polynomial of 
with  $\deg(Q_1)\leq n$ and
$$Q_1(P_{m,n}) = \tau^*(Q_0)(P_{m,n})
= Q_0(\tau(P_{m,n}))=0\,, $$
which is a contradiction to $P_{m,n}$ being unisolvent. Therefore, $\tau(P_{m,n})$ must be unisolvent.
\qed\end{proof}

Next we use the ingredients above to decompose the PIP.

\subsection{Decomposition of the Multivariate PIP onto Hyperplanes}
Section \ref{SC} provides  the ingredients to prove our first key result. 
To avoid confusion, we note that a $0$-dimensional hyperplane $H \subseteq \R^m$, $m \geq 1$, is given by a single point $H=\{\mathrm{point}\}$, such that
the PIP on $H$ is solved by evaluating $f$ at that point. Vice versa, the zeroth-order PIP (i.e., a constant) with respect to $f : \R^m\lo \R$ and $n=0$ is solved by evaluating $f$ at some point 
$q \in \R^m$.

\begin{theorem} Let $m,n \in \N$, $m \geq 1$, and $P \subseteq \R^m$, such that:
\begin{enumerate}
 \item[i)] There exists a hyperplane $H \subseteq \R^m$ of co-dimension 1 such that $P_1:= P \cap H$ satisfies $\#P_1 = N(m-1,n)$ and is unisolvent with respect to $H$, i.e., 
 by identifying $H\cong \R^{m-1}$ the Vandermonde matrix $V_{m-1,n}(P_1)$ is regular.
  \item[ii)]  The set $P_2= P \setminus H$ satisfies $\#P_2 = N(m,n-1)$ and is unisolvent with respect to $\mbox{(m,n-1)}$, i.e., the Vandermonde matrix $V_{m,n-1}(P_2)$ is regular.
\end{enumerate}
Then $P$ is a unisolvent node set. 
\label{GN}
\end{theorem}

\begin{proof} 
Due to Lemma \ref{TT}, unisolvent node sets remain unisolvent under affine transformation. Thus, by choosing the appropriate transformation $\tau$,  we can assume  w.l.o.g.~that $H= H_{1,\dots, m-1}$. 
For any polynomial $Q \in \R[x_1,\dots,x_m]$ there holds 
$Q\big(\pi_{1,\dots,m-1}(p)\big) = \pi_{1,\dots,m-1}^*\big(Q(p)\big)$ for all $p \in H$. Thus,  by $(i)$ we observe that  
whenever there is a $Q \in \R[x_1,\dots,x_m]$ with $Q(P_1) =0$, then $\deg(Q) \geq \deg\big(\pi_{1,\dots,m-1}^*(Q)\big) > n$. Or, if $\deg\big(\pi_{1,\dots,m-1}^*(Q)\big) \leq \deg(Q) \leq n$, we consider the polynomial 
$\bar Q_1:= Q - i^*_{1,\dots,m-1}\big(\pi_{1,\dots,m-1}^*(Q)\big)$, which consists of all monomials sharing the variable $x_m$ and $\bar Q_2:= Q -\bar Q_1$ consisting of all monomials not sharing $x_m$.  
We claim that $\bar Q_2 = 0$. Certainly, $\bar Q_1(x) = 0$ for all $x \in H$.  Since $P_1$ is unisolvent there are $p \in P_1$ with $\bar Q_2(p)= \pi_{1,\dots,m-1}^*(\bar Q_2(p)) \not= 0$ implying $Q(p) \not = 0$,
which contradicts our assumption on $Q$ and therefore yields $\bar Q_2 = 0$ as claimed. 
In light of this fact, 
$Q$ can be decomposed into polynomials $Q=Q_1\cdot Q_2 \quad \text{where} \quad   Q_2(x_1,\dots,x_m) =x_m\,.$ Since $\deg(Q_1)\leq n-1$ and $P_2$ is unisolvent we have that $P_2 \not \subseteq Q_1^{-1}(0)$.
At the same time, $P_2 \cap H =\emptyset $ implies that $Q_2(p) \not = 0$ for all $p \in P_2$. Hence, there is $p \in P_2$ with $Q(p)\not =0$, proving the theorem due to Proposition \ref{generic}. 
\qed\end{proof}

The question arises whether the decomposition of unisolvent nodes given in Theorem \ref{GN} allows us to decompose the PIP into smaller, and therefore simpler, sub-problems.
Indeed we obtain: 

\begin{theorem} \label{SV} Let $m,n \in \N$, $m \geq 1$,  $f : \R^m\lo \R$ be a computable function, \linebreak 
$H \subseteq \R^m$ a hyperplane of co-dimension 1, $Q_H \in \Pi_{m,1}$ a polynomial satisfying $Q_H^{-1}(0) =H$, 
and $P_{m,n}=P_1\cup P_2\subseteq \R^m$ such that $(i)$ and $(ii)$ of Theorem \ref{GN} hold with respect to $H$.
Require $Q_1,Q_2 \in \Pi_{m,n}$ to be such that: 
\begin{enumerate}
 \item [i)] $Q_1$  solves the PIP with respect to $f$ and $P_1$ on $H$.
 \item [ii)] \mbox{$\deg(Q_2) \leq n-1$} and $Q_2$  solves the PIP with respect to \mbox{$(m,n-1)$}, $f_1:=(f-Q_1)/Q_H$, and $P_2=P\setminus H$ on $\R^m$.
\end{enumerate}
Then, $Q_{m,n,f} =Q_1+Q_HQ_2$ is the uniquely determined polynomial of $\deg(Q)\leq n$ that solves the PIP with respect to $f$ and $P_{m,n}$ on $\R^m$.
\end{theorem}
\begin{proof}  By our assumptions on $Q_H$ and $Q_1$ we have that $Q(x)=Q_1(x)$ $\forall x \in H$ and therefore $Q(p) =f(p)$ $\forall p \in P_1 = P_{m,n} \cap H$. 
At the same time, $Q_H(x) \not =0$ for all $x \not \in H$. Therefore, $f_1$ is well defined on $P_2$ and $Q(p)= Q_1(p) + f_1(p)Q_H(p)=f(p)$ $\forall p \in P_2=P_{m,n}\setminus H$. 
Hence, $\deg(Q)\leq n$, and $Q$ solves the PIP with respect to $f$ 
and a unisolvent set of nodes $P_{m,n}\subseteq \R^m$. Thus, $Q_{m,n,f}$ is the unique solution of the PIP with respect to $P_{m,n}$ and $f$. 
 \qed\end{proof}
 
 \begin{remark}
 Note that $Q_H$ can be constructed by choosing a (usually unit) normal vector $\nu \in \R^m$ onto $H$ and a vector $b \in H$, setting 
 $$Q_H(x) = \nu\cdot(x - b)\,.$$ Indeed, this guarantees that $Q_H(x)=0$ for all $x \in H$ and 
 $Q_H(x)\not =0$ for all $x \in \R^m\setminus H$.
  \end{remark}
  
As a consequence of  Theorems \ref{GN} and \ref{SV}, we can close this section by stating the first part of Main Result I and delivering its proof. 
 \begin{theorem}Let $m,n \in \N$, and $f : \R^m \lo \R$ be a computable function. Then, there exists an algorithm with runtime complexity in $\Oc\big(N(m,n)^2\big)$,  
 requiring storage in $\Oc\big(N(m,n)\big)$, that computes:
 \begin{enumerate}
  \item[i)] A unisolvent node set $P_{m,n} \subseteq \R^m$.
  \item[ii)] The  normal form coefficients $c_0,\dots,c_{N(m,n)-1}$ of the interpolation polynomial $Q_{m,n,f} \in \Pi_{m,n}$ with $Q_{m,n,f}(p) = f(p)\,, \,\forall p \in P_{m,n}$.
 \end{enumerate}
\label{mainT}
 \end{theorem}
 
 \begin{proof} We start by proving $(i)$ and $(ii)$ with respect to the runtime complexity. To do so, we claim that there is a constant $C\in \R^+$ and an algorithm computing $(i)$ and $(ii)$ in less than \mbox{$CN(m,n)^2$} 
 computation steps. To prove this claim, we argue by induction on $N(m,n)$.
 If $N(m,n)=1$, then $m=0$ or $n=0$. Thus, interpolating $f$ is given by evaluating $f$ at one single node, which can be done in $\Oc(1)$ by our assumption on $f$. 
 
 Now let $N(m,n)>1$. Then $m>0$ and we choose $\nu,b\in \R^m$ with $\nol\nu \nor=1$ and consider the hyperplane $H=Q_H^{-1}(0)$, $Q_H(x)=\nu(x-b)$.
 By identifying $H\cong \R^{m-1}$, induction yields that we can determine a set of unisolvent nodes $P_1\subseteq H$ in less than \mbox{$CN(m-1,n)^2$} computation steps. 
 Induction also yields that a set $P_2 \subseteq \R^m$ of unisolvent nodes can be determined with respect to $n-1$ in less than
 \mbox{$CN(m-1,n)^2$} computation steps. By translating $P_2$ with $\lambda \nu$, i.e., setting $P_2'=P_2 + \lambda \nu $, $\lambda \in \R$, we can guarantee that 
 $P_2 \cap H = \emptyset$. Hence, the union $P_{m,n}=P_1 \cup P_2'$ of the corresponding unisolvent sets of nodes is also unisolvent due to Theorem \ref{GN}, proving $(i)$. 
 
To show $(ii)$ we have to compute the coefficients of $Q_{m,n,f}$ in normal form. 
By induction, a  polynomial $Q_1 \in \Pi_{m,n}$ that solves the PIP with respect to $f$ and $P_1$ on $H$ can be determined in  normal form in less than $CN(m-1,n)^2$ steps. 
We consider $f_1' : \R^m \lo \R$ with $f_1'(x):= \big(f(x-\lambda \nu) - Q_1(x-\lambda \nu)\big)/Q_H(x-\lambda \nu)$. By induction, we can compute $Q_2\in\Pi_{m,n-1}$  
in less than $CN(m,n-1)^2$ computation steps, such that $Q_2$ solves the PIP with respect to $f_1'$ and $P_2$, \mbox{$(m,n-1)$}. 
Thus, $Q_2$ solves the PIP with respect to 
$f_1(x)= \big(f(x) - Q_1(x)\big)/Q_H(x) = f_1'(x+\lambda \nu)$, $P_2'$, and \mbox{$(m,n-1)$}.
Due to Theorem \ref{SV}, we have that $Q_1+Q_HQ_2$ solves the PIP with respect to $f$ and $P_{m,n}$. 
It remains to bound the steps required for computing the normal form of $Q_1+Q_HQ_2$. The bottleneck herein lies in the computation of $Q_HQ_2$, which requires $C_2(m+1)N(m,n-1)$, $C_2 \in \R^+$ 
 computation steps. Observe that $m \leq N(m-1,n)$ for $n\geq 1$, which shows that $Q_1+Q_HQ_2$ can be computed in less than $C_3N(m-1,n)N(m,n-1)$, $C_3 \in \R^+$ computation steps. It is $N(m,n)=N(m-1,n)+N(m,n-1)$. 
 Hence, by  assuming $2C\geq C_3$, we have that 
 $$ C\big(N(m-1,n)^2 + N(m,n-1)^2\big) + C_3N(m-1,n)N(m,n-1) \leq CN(m,n)^2 \,,$$
 proving $(ii)$. 
 
 We obtain the storage complexity by using an analogous induction argument. If $N(m,n)=1$ then $m=0$ or $n=0$ and $f: \{\mathrm{point}\}\lo \R$ is interpolated by evaluating $f$ at one point. 
 Therefore we have to store that point yielding storage complexity $\Oc(1)$. 
  If  $N(m,n)>1$, using the same splitting of the problem as above, induction yields that there is $D \in \R^+$ such that we have to store at most $DN(m-1,n)$ and $DN(m,n-1)$ numbers 
  for each of the two sub-problems, respectively. Altogether we need to store 
$D\big(N(m-1,n) +N(m,n-1)\big) = DN(m,n)$ numbers, proving the storage complexity. 
  \qed\end{proof}

So far, we have provided an existence result stating that the PIP can be solved efficiently in $\Oc(N(m,n)^2)$.
The derivation of an actual algorithm based on the recursion implicitly used in Theorem \ref{mainT} is given in the next section.

\section{Recursive Decomposition of the PIP}\label{RPIP}   
Based on the recursion expressed in Theorems \ref{GN} and \ref{SV}, we can derive an efficient algorithm to compute solutions to the PIP in a numerically robust and computationally efficient way. The resulting algorithm derived hereafter, called PIP-SOLVER is based on a binary tree $T_{m,n}$ as a straightforward data structure resulting from the recursive problem decomposition.
It also uses a \emph{multivariate divided difference scheme} to solve the sub-problems, leading to better numerical robustness and accuracy than other approaches (comparisons will be given in Section \ref{EX}).

\subsection{Multivariate Newton Polynomials }
The essential data structure required to formulate multivariate Newton interpolation is given by the following binary tree: 
\begin{definition}[PIP tree] Let $m,n \in \mathbb{N}$. We define a binary tree \linebreak $T_{m,n}=(\Vc_{m,n},\Ec_{m,n})$
with vertex labeling $\sigma : \Vc_{m,n}\lo \N\times \N$, $\sigma(v)=(\dim(v),\deg(v))$ as follows: 
We start with a root $\rho$ and set \mbox{$\dim(\rho)=m$}, \mbox{$\deg(\rho)=n$}. For a vertex $v \in \Vc_{m,n}$ with  \mbox{$\dim(v) >0$} and \mbox{$\deg(v)> 0$} we 
introduce a left child $l$ with label \mbox{$\sigma(l)=(\dim(v)-1,\deg(v))$} and a right child $r$ with label  \mbox{$\sigma(r)=(\dim(v),\deg(v)-1)$}.

Furthermore, we consider the set $\Gamma_{m,n}$ of \emph{leaf paths} $\gamma$ starting at $\rho$ and ending in a leaf $v_\gamma \in L_{m,n} \subset \Vc_{m,n}$ of $T_{m,n}$. 
For every leaf path $\gamma \in \Gamma_{m,n}$ we define
\begin{equation}
 [\gamma]=(k_1,\dots,k_m)^\mathsf{T} \in \N^{m}\,, \quad k_i = \# \li\{v \in \gamma\mi \dim(v)=i\re\}
\end{equation}
to be its \emph{descent vector}, 
which uniquely identifies the leaf path $\gamma$.
\label{TMN}
\end{definition}
Figure \ref{TMN} illustrates the tree $T_{m,n}$ and the concept of a leaf path $\gamma$.  As one can easily verify, the depth of $T_{m,n}$ is given by $\mathrm{depth}(T_{m,n})=m+n-1$, while the  
total number $\# L_{m,n}$ of leaves of $T_{m,n}$ is given by $N(m,n)$. 

\begin{figure}[t!]
\center
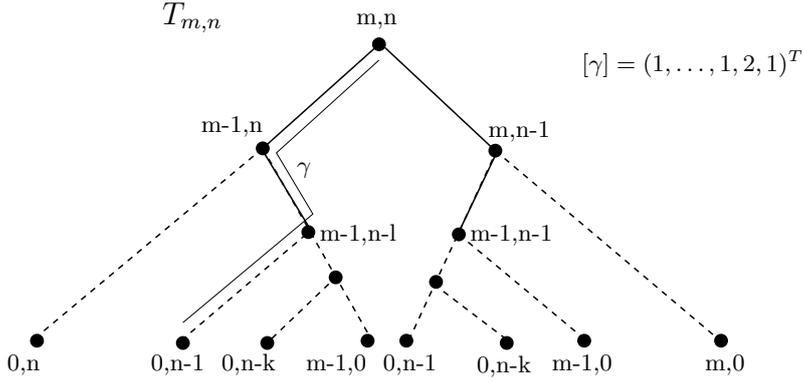
\caption{The binary tree $T_{m,n}$ and a leaf path $\gamma\in \Gamma_{m,n}$ with descent vector $[\gamma]$.}
\end{figure}

Next we introduce the multidimensional generalization of Newton nodes, given by a \emph{non-uniform, affine-transformed, sparse T-grid}:

\begin{definition}[Multidimensional Newton nodes] Given $m,n \in \N$ and a set of nodes  $P_{m,n}\subseteq \R^m$ with $\#P_{m,n}=N(m,n)$, we say $P_{m,n}$ are {\em multidimensional Newton nodes} generated by 
$$\PP_{m,n}: = \oplus_{i=1}^m\PP_i\,, \quad \PP_i=\{p_{i,1},\dots,p_{i,n+1}\} \subseteq \R$$ if and only if 
there exists  an affine transformation $\tau : \R^m \lo \R^m$ 
such that for every $p \in P_{m,n}$ there is exactly one  $\gamma \in \Gamma_{m,n}$ with $[\gamma]=(k_1,\dots,k_m)^T\in \N^m$ such that: 
$$ \tau(p)= (\bar p_{1,k_1},\dots,\bar p_{m,k_{m}})^\mathsf{T}\,, \,\text{where we set } \,\, \bar p_{k_l}= \li\{\begin{array}{ll}
                                                                                                     p_{l,n+1} &\text{if}\,\, k_{l'}=0 \,,\,\,\forall \, l'<l\\
                                                                                                     p_{l,k_l} & \text{else}\,.
                                                                                                    \end{array}\re.
$$
If $\tau=\id_{\R^m}$, we call $P_{m,n}$ \emph{canonical multidimensional Newton nodes}.
\label{NTN}
\end{definition}
Considering the special case of canonical multidimensional Newton nodes, and assuming $\PP_i =\PP_{i'}$ for all occurring $1 \leq i,i' \leq m$, an intuition of the definition in terms of sparse $T$-grids reveals. Indeed, $T$-grids are based on decomposing space along a binary tree, as we do here. Furthermore, the tree $T_{m,n}$ is a sparse binary tree, i.e., the full tree depth is only reached by few of the leaves.

Note that the $P_{m,n}$ can be recursively split into subsets $P_1,P_2 \subseteq P_{m,n}$ with $P_1\subseteq H$ for some hyperplane $H\subseteq \R^m$, such that the assumptions of Theorem \ref{GN} are satisfied for every recursion step. Even more, the recursion satisfies the \emph{defining property of Newton nodes}, which is that the projection 
\begin{equation} \label{parallel}
 \pi_H(P_2) \subseteq H \,\,\, \text{of $P_2$ onto}\,\,\, H\,\,\, \text{satisfies}\,\,\, \pi_H(P_2)\subseteq P_1\,.
\end{equation}
This fact is going to be the key ingredient for the proof of Theorem \ref{DDST}.
 
\begin{definition}[Multivariate Newton polynomials] Given $m,n \in \N$, $T_{m,n}$, $\Gamma_{m,n}$, and multidimensional Newton nodes $P_{m,n}$ generated by $\PP_{m,n}$.
For $\gamma \in \Gamma_{m,n}$ with descent vector $[\gamma]=(k_1,\dots,k_m)^\mathsf{T}\in \N^{m}$ and $x=(x_1,\dots,x_m)^\mathsf{T}\in \R^m$, we call 
\begin{equation}
 N_{[\gamma]}(x)= \prod_{i=1}^m\prod_{l=1}^{k_i-1}(x_i-p_{i,l})
\end{equation}
the {\em multivariate Newton polynomials} with respect to $P_{m,n}$.
\end{definition}

\begin{proposition} Let $m,n \in \N$, $f: \R^m \lo \R$ be a function, and $P_{m,n}$ a set of multidimensional Newton nodes generated by $\PP_{m,n}$. Then, there exist unique coefficients $c_{[\gamma]}\in \R$, $\gamma \in \Gamma_{m,n}$, such that: 
$$ Q(x) = \sum_{\gamma \in \Gamma(m,n)} c_{[\gamma]} N_{[\gamma]}(x)$$ 
satisfies $Q(p)=f(p)$ for all $p \in P_{m,n}$. Since $Q \in \Pi_{m,n}$, $Q$ is the unique solution of the PIP with respect to $(m,n,f,P_{m,n})$, implying that $P_{m,n}$ is unisolvent.
\label{UniN} 
\end{proposition}

\begin{proof} We argue by induction on $N(m,n)$. If $N(m,n)=1$ then $P_{m,n}=\{p\}$ consists of one point. Therefore, $Q(x)=f(p)$ and the claim holds. If $N(m,n)>1$, then we can assume w.l.o.g.~that $P_{m,n}$ are canonical, i.e., that the affine transformation $\tau$ in Definition \ref{NTN} is the identity, $\tau = \id_{\R^m}$ (see Lemma \ref{TT} and Definition \ref{NTN}). 
We consider 
$$P_1=\li\{p=(p_1,\dots,p_{m})^\mathsf{T} \in P_{m,n} \subseteq \R^m \mi p_m=p_{m,1}\in \PP_{m}\re\}\,,$$ 
$P_2 =P_{m,n}\setminus P_1$, and $H= \li\{x \in \R^m \mi x_m =p_{m,1}\re\}$. Now note that
$P_1$ is a multidimensional Newton node set on $H$ w.r.t.~\mbox{$(m-1,n)$} and $P_2$ is a multivariate Newton node set w.r.t.~\mbox{$(m,n-1)$} on $\R^m$. Hence, by induction, $P_1,P_2$ are unisolvent.  
Thus, $P_1,P_2$, and $H$ satisfy the assumptions of Theorem \ref{GN}, which completes the proof.
\qed\end{proof}
Indeed, in dimension $m=1$ the $N_{[\gamma]}(x)$ correspond to the classical Newton polynomials from Eq.~\eqref{NewtPoly}. Given canonical multidimensional Newton nodes $P_{m,n}$, and ordering the paths $\Gamma_{m,n}$ 
from left to right with respect to the leaf ordering of  $T_{m,n}$, we observe that 
for the multidimensional Newton nodes $p_{[\gamma]}$ corresponding to $\gamma$
there holds 
\begin{equation}
 N_{[\gamma]}(p_{[\gamma']})= 0 \quad \text{if}\,\,\, \gamma < \gamma'\,.
\end{equation}
Thus, like in the 1D case, the associated Vandermonde matrix $V_{m,n}(P_{m,n})$ is of lower triangular form, yielding an alternative proof of Proposition \ref{UniN} and a  possibility of computing the coefficients $c_{[\gamma]}$ in $\Oc(N(m,n)^2)$ steps.
Just like in the 1D case, a divided difference scheme can be formulated for computing the coefficients $c_{[\gamma]}$ even more efficiently.

\subsection{Alternative Formulation of the 1D Divided Differences Scheme}

\begin{figure}[t!]
\center
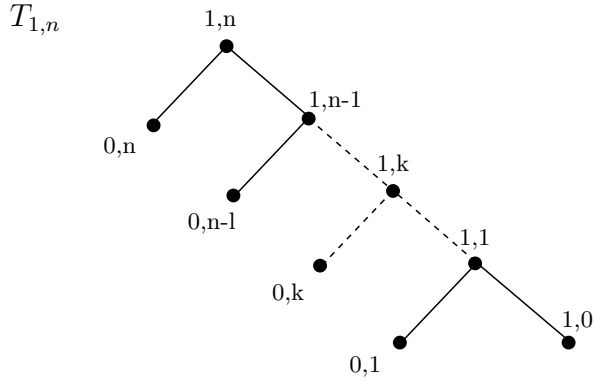
\caption{The binary tree $T_{1,n}$ in the 1-dimensional case.}
\label{T1D}
\end{figure}

We consider the  case of dimension $m=1$. In this case, the tree $T_{1,n}$ has the special form illustrated in Figure \ref{T1D}. Choosing pairwise disjoint nodes $P_n=\{p_1,\dots,p_{n+1}\}\subseteq \R$ 
and enumerating the leaf paths $\gamma _i  \in \Gamma_{1,n}$ from left to right, we observe that
 $$N_{[\gamma_i]}=N_{i}(x)=\prod_{j=1}^i (x_j-p_j)  \qquad i =0,\dots,n\,,$$ 
 where the Newton polynomials $N_i$ were defined in Eq.~\eqref{NewtPoly}. Though it is trivial, the nodes $P_n$ can be seen as lying on the $0$-dimensional planes $H_i=\{x \in \R^1 \mi x= p_i\}=\{p_i\}$.
 While Theorem \ref{GN} is fulfilled anyhow in this case, this interpretation allows to recursively set up the statement of Theorem \ref{SV}. 
 That is, for a given computable function $f: \R\lo\R$ and $k=1,\dots,n$ we set:
\begin{align}
 f_0(x)&=f(x)\,,  & c_0&=f_0(p_1) \,, \nonumber \\
f_k(x)&=\frac{f_{k-1}(x)-f_{k-1}(p_{k})}{x-p_k} \,,  & c_k&=f_k(p_{k+1})\,. \label{DDS2}
\end{align}
Indeed the computation of the $c_k$ is similar to the divided difference scheme illustrated in Eq.~\eqref{DDS}. However, the two schemes are not identical.  
In our alternative formulation, the differences and divisors appear in a different order. Still, both schemes compute the same result, as 
can easily be verified by hand. 

\begin{corollary} Let $f : \R \lo \R$ be a computable function, and $T_{1,n}$, $\Gamma_{1,n}$ be given. Further let $P_n=\{p_1,\dots,p_{n+1}\} \subseteq \R$ be a set of pairwise disjoint nodes. 
Then, the polynomial 
$$ Q(x)=\sum_{k=0}^{n}c_kN_{[\gamma_k]}(x)$$
that solves the 1-dimensional PIP with respect to \mbox{$(n,f,P_n)$} can be determined in $\Oc(n^2)$ computation steps. 
\label{DDSC}
\end{corollary}
\begin{proof} By induction on $n$, the statement follows directly from Proposition \ref{UniN} and Theorem \ref{SV}. The runtime 
complexity can be shown analogously to the proof of Theorem \ref{mainT}.
\qed\end{proof}

\subsection{The Multivariate Divided Differences Scheme} 
In the general case of dimension $m \in \N$, we introduce the following notions:

\begin{definition}[parallel paths] Let $m,n \in \N$, and $T_{m,n}$, $\Gamma_{m,n}$ be given. Let $\gamma \in \Gamma_{m,n}$  with descent vector $[\gamma]=(k_1,\dots,k_m)\in \N^m$. We consider 
$1\leq i\leq m$ with $k_i >1$. 
If the leaf $l_\gamma$ where the path encoded by $\gamma$ ends has degree $\deg(l_\gamma)>0$, then we define
\begin{align*}
 \Delta_i(\gamma):=
  \li\{\delta \in \Gamma_{m,n} \mi [\delta]=(l_1,\dots,l_m) \,\,\text{with} \,\,\begin{array}{ll}
                                                                                       1\leq l_i < k_i \,, &   \\ 
                                                                                       l_j = k_j\,, & \text{for} \,\, j\not = i 
                                                                                       \end{array}\,\,  \re\} .
                                                                                       \end{align*}
If the leaf degree $\deg(l_\gamma)=0$ then for  $i=1$ we consider 
$$  \Delta_1^0(\gamma):=
  \li\{\delta \in \Gamma_{m,n} \mi [\delta]=(l_1,k_2\dots,k_m) \,, \quad l_1 =1,\dots,k_1-1\re\}
  $$
and for $i >1$ we define $i_0=\min \li(\{j < i \mi k_j >1\} \cup \{1\}\re)$ and
\begin{align*}
 \Delta_i^0(\gamma):=
  \li\{\delta \in \Gamma_{m,n} \mi [\delta]=(l_1,\dots,l_m) \,\text{with} \,\begin{array}{ll}
										       l_{i} < k_{i}\,, & \\
                                                                                       l_j = k_j\,, & j\not = i,i_0 \\                                                                                      
                                                                                       l_{i_0} = k_{i}-l_i+k_{i_0}\,, &  
                                                                                       \end{array}\,\,  \re\}\,.
                                                                                       \end{align*}
Finally, we order the paths in $\Delta_i(\gamma),\Delta_i^0(\gamma)$ with respect to the  leaf ordering of $T_{m,n}$ from left to right.
\end{definition}

Indeed, the $\Delta_i(\gamma),\Delta^0_i(\gamma)$ correspond to all multidimensional Newton nodes that are required to compute the coefficient of $N_{[\gamma]}$. For instance, 
the leafs $l_\delta$ of $ \delta \in \Delta_{i}^0(\gamma)$, $i>1$, have degree $\deg(l_\delta)=0$ and are given by parallel translation of $l_\gamma$ along the multidimensional Newton grid $P_{m,n}$ in dimension $i$. 
More precisely, we define:

\begin{definition}[multivariate divided differences] Let $m,n \in \N$, \linebreak 
$f : \R^m \lo \R$ be a computable function, $P_{m,n}$ be a set of multidimensional Newton nodes, and $T_{m,n},\Gamma_{m,n}$ be given. 
For $\gamma \in \Gamma_{m,n}$ with $[\gamma]=(k_1,\dots,k_m)$,  
we consider $\Delta_i(\gamma),\Delta_i(\gamma)^0=\{\delta_{1,i},\dots,\delta_{k_i,i}\}$, and $m_\gamma \in \N$ such that  $J_{\gamma}=\{ j_1,\dots,j_{m_{\gamma}}\} \subseteq \{1,\dots,m\}$ contains all indices
with $k_{i} >1$ for $i \in J_{\gamma}$. By reordering if necessary we assume that $J_\gamma$ is ordered in reverse direction, i.e.,  $ j_{l} > j_{l'}$ if $l<l'$ for all
$l,l' \in \{1,\dots,m_{\gamma}\}$.
We denote by $p_{h,i}\in P_{m,n}$ the nodes corresponding to $\delta_{h,i}$ and define
\begin{align*}
F_{\gamma,1}^0 & =f(p_{j_1,1}) \,, & F_{\gamma,1}^k &:=\frac{F_{\gamma,1}^{k-1}(p_{j_1,k}) - F_{\gamma,1}^{k-1}(p_{j_1,k-1})}{(p_{j_1,k-1}-p_{j_1,k})_{j_1} }\,,  \,\,\,1\leq k\leq k_{j_1}\,, \\
F_{\gamma,i}^0 & :=F_{\gamma,i-1}^{k_{j_i}} \,, & F_{\gamma,l}^k &:=\frac{F_{\gamma,l}^{k-1}(p_{j_l,k}) - F_{\gamma,l}^{k-1}(p_{j_l,k-1})}{(p_{j_l,k-1}-p_{j_l,k})_{j_l} }\,,
 \begin{array}{l}
 1\leq k\leq k_{j_l}\,,\\
 1\leq l \leq m_\gamma\,,
 \end{array} 
\end{align*}
where $(p_{j_l,k-1}-p_{j_l,k})_{j_l}$ denotes the $j_l$-th coordinate/component of $(p_{j_l,k-1}-p_{j_l,k})$.
\label{MVDD} 
\end{definition}

This definition yields the alternative divided difference scheme given in Eq.~\eqref{DDS2} for the special case $m=1$.

\begin{theorem} Let $m,n \in \N$, $f : \R^m \lo \R$ be a computable function, $T_{m,n}$, $\Gamma_{m,n}$ be given, and $P_{m,n}\subseteq \R^m$  be a set of canonical multidimensional Newton nodes generated by $\PP_{m,n}$. 
Then, the polynomial 
$$ Q(x)=\sum_{\gamma \in \Gamma_{m,n}}c_{[\gamma]}N_{[\gamma]}(x)\,, \quad c_{[\gamma]}= F_{\gamma,m_\gamma}^{k_{j_{m_\gamma}}}\,, $$
with $[\gamma]=(k_1,\dots,k_{j_{m_\gamma}},\dots,k_{j_1},\dots,k_m)$ solves the PIP with respect to $(m,n,f,$ $P_{m,n})$ and can 
be computed in $\Oc(N(m,n)^2)$ operations.
\label{DDST} 
\end{theorem}
\begin{proof} We argue by induction on \mbox{$(m,n)$}. For $m=1, n\in \N$, the statement is already proven in Corollary \ref{DDSC}. Thus, we proceed by induction on $m$. 
If $m>1$, then 
we consider $\Gamma_1=\{\gamma_1,\dots,\gamma_{N(m-1,n)} \}\subseteq  \Gamma_{m,n}$ with 
$[\gamma_h]=(k_1,\dots,k_m)\in \N^{m} $ and 
$k_{m}=1$, $h=1,\dots,N(m-1,n)$. Furthermore, we denote by $P_1\subseteq P_{m,n}$ the corresponding nodes and $\Gamma_2 = \Gamma_{m,n}\setminus \Gamma_1$, $P_2 = P_{m,n}\setminus P_1$. In addition, we consider the sub-trees
$T_1\cong T_{m-1,n}$, $T_2\cong T_{m,n-1}$ spanned by $\Gamma_1$, $\Gamma_2$, respectively.
Note that the $m$-th coordinate is constant, i.e., for $p,p'\in P_1$ we have $p_m=p_m'=p_{m,1} \in \PP_{m,n}$. Since the hyperplane 
$H= \li\{x \in \R^m \mi x_m =p_{m,1}\re\}$  can be identified with $\R^{m-1}$ by shifting the last coordinate to $0$, induction yields that 
$$c_{[\gamma_1]} = F_{\gamma_1,m_{\gamma_1}}^{k_{j_{m_{\gamma_1}}}},\dots,c_{[\gamma_{N(m-1,n)}]}= F_{\gamma_{N(m-1,n)},m_{\gamma_{N(m-1,n)}}}^{k_{j_{m_{\gamma_{N(m-1,n)}}}}}\,.$$
These are the uniquely determined coefficients solving the 
PIP with respect to \mbox{$(m-1,n,f_{| H},P_1)$} on $H$. We further follow Theorem \ref{SV} and set $f_1(x)=(f(x)-Q_1(x))/(x_m-p_{m,1})$ with $Q_1(x) = \sum_{\gamma \in \Gamma_1}c_{[\gamma]}N_{[\gamma]}(x)$.
The defining property of multidimensional Newton nodes is their parallelism, i.e., 
$\pi_H(P_{2}) \subseteq P_1$, where $\pi_H$ denotes the orthogonal projection onto $H$. Furthermore, we observe that $Q_1(x)$ is constant in direction perpendicular to $H$. Therefore, 
$$ Q_1(p) = Q_1(\pi_H(p))=f(\pi_H(p)) \quad \text{for all}\quad p \in P_2\,.$$ 
Thus, 
\begin{align*}
 f_1(p) &= (f(p)-f(\pi_H(p)))/(\pi_H(p)_m - p_{m,1})\\
 &=(f(p)-f(p_1,\dots, p_{m-1},p_{m,1}))/(p_m - p_{m,1})\,.
\end{align*}
Following Definition \ref{MVDD}, we observe that  
$
d_{[\delta_1]} = F_{\delta_1,K_{\delta_1}}^{k_{j_{m_{\delta_1}}}},\dots,d_{[\delta_{N(m-1,n)}]}$ 
\linebreak 
$= F_{\delta_{N(m,n-1)},K_{\delta_{N(m,n-1)}}}^{k_{j_{m_{\delta_{N(m,n-1)}}}}}  
$
yielding 
\begin{equation}
 d_{[\delta_1]}=c_{[\gamma_{N(m-1,n)+1}]}\,\dots,d_{[\delta_{N(m-1,n)}]}= c_{[\gamma_{N(m,n)}]}\label{ce}
\end{equation}
Hence, by induction, the $d_{[\delta]}$, $\delta\in \Gamma_{m,n-1}$, are the uniquely determined coefficients solving the 
PIP w.r.t.~\mbox{$(m,n-1,f_{1},P_2)$} on $\R^m$. Moreover, due to Eq.~\eqref{ce}, we have: 
$$Q_2(x) = \sum_{\delta\in T_{m-1,n}}d_{[\delta]}N_{[\delta]}(x) = \frac{1}{(x_m-p_{m,1})}\sum_{\gamma\in \Gamma_{2}}c_{[\gamma]}N_{[\gamma]}(x)\,.$$
Now Theorem \ref{SV} yields that $$ Q(x)= Q_1(x)+Q_H(x)Q_2(x)= \sum_{\gamma \in \Gamma_{m,n}}c_{[\gamma]}N_{[\gamma]}(x)\,, \quad   Q_H(x) = x_m-p_{m,1}, $$ solves the PIP w.r.t.~\mbox{$(m,n,f,P_{m,n})$}, which proves the
statement. The runtime complexity follows from the proof of Theorem \ref{mainT}.
\qed\end{proof}

By considering $f'=f \circ \tau^{-1}$ with $\tau$ from Definition \ref{NTN}, Theorem \ref{DDST} also extends to the case of non-canonical multidimensional Newton nodes, i.e., to affine transformations of canonical multidimensional Newton nodes. 
Together, Theorems \ref{SV} and \ref{DDST} and Proposition \ref{UniN}  
prove Main Result I as stated in Theorem \ref{I} in Section \ref{main}.
We next characterize some basic properties of multivariate Newton polynomials and how they can be exploited for basic computations.

\subsection{Properties of Multivariate Newton Polynomials} Analogously to its 1D version, the multivariate divided difference scheme provides an efficient and numerical robust method for  interpolating functions by polynomials.
We generalize some classical facts of the 1D case, yielding Main Result \ref{II}, here restated with a bit more detail as: 
  \begin{theorem}[Main Result II]Let $m,n \in \N$, $T_{m,n}$, $\Gamma_{m,n}$, and a  set of canonical multidimensional Newton nodes $P_{m,n}$ be given. Let further $Q \in \Pi_{m,n}$, 
  $Q(x)= \sum_{\gamma \in \Gamma_{m,n}}c_{[\gamma]}N_{[\gamma]}(x)$ be a polynomial given in multivariate Newton form w.r.t.~$T_{m,n}$, $P_{m,n}$. Then, there exist algorithms with runtime complexity in $\Oc(N(m,n))$ that compute: 
  \begin{enumerate}
   \item[i)] The value of $Q(x_0)$  for any $x_0 \in \R^m$  in 
  $\Oc(N(m,n))$;   
   \item[ii)] The partial derivative  $\partial_{x_i} Q|_{x_0}$ for any $i \in \{1,\dots,m\}$ and $x_0 \in \R^m$ in 
  $\Oc(nN(m,n))$; 
     \item[iii)]  The integral $\int_{\Omega}Q(x)\mathrm{d}x$ for any hypercube $\Omega \subseteq \R^m$ with runtime complexity in 
  $\Oc(nN(m,n))$.
   \label{M2}
  \end{enumerate}
  \label{EVN}
\end{theorem}

  \begin{proof} We argue by induction on $m,n \in \N$. For $m=1, n\in \N$ all statements are known to be true \cite{Stoer,gautschi,atkinson,endre}.  If $m>1$ then, by assumption, the affine transformation 
   $\tau : \R^m \lo \R^m$ from Definition \ref{NTN} is the identity, i.e., $\tau=\id_{\R^m}$.  Let $\PP_{m,n}:=\oplus_{i=1}^m \PP_i$, $\PP_i=\{p_{i,1},\dots,p_{i,n+1}\}$ be the generating nodes of $P_{m,n}$, see  Definition \ref{NTN}.
   We consider the hyperplane 
  $H=\li\{x \in \R^m \mi x_m = p_{m,1}\re\}$ and, for $\Gamma_1=\{\gamma\in \Gamma_{m,n}\mi [\gamma]_{m}=1\}$, we denote by $P_1 \subseteq P_{m,n}$ the multidimensional Newton nodes corresponding to the leaves of $\gamma \in \Gamma_1$. 
  Further, we set $Q_1(x)= \sum_{\gamma \in \Gamma_1}c_{[\gamma]}N_{[\gamma]}(x)$. By 
  identifying $H$ with $\R^{m-1}$ induction yields that $(i),(ii),(iii)$ can be computed in $\Oc(N(m-1,n))$. Setting $Q_2(x) = Q(x)-Q_1(x)$, 
  we have $Q_2(x) = (x_m-p_{m,1})Q_3(x)$, where $Q_3 \in \Pi_{m,n-1}$ is
  a polynomial in Newton form w.r.t.~$T_{m,n-1}$ and $P_2=P_{m,n}\setminus P_1$ (see Theorem~\ref{DDST}). Thus, by induction, $(i),(ii)$ can be computed in $\Oc(N(m,n-1))$, $\Oc(nN(m,n-1))$ for $Q_3$, which proves $(i),(ii)$ 
  because $N(m-1,n)+N(m,n-1)=N(m,n)$. Claim $(iii)$ follows 
  from $(i)$ and $(ii)$ by applying partial integration to $Q(x) = Q_1(x) + (x_m-p_{m,1})Q_3(x)$, i.e., 
  \begin{align*}
   \int_{\Omega} Q(x)dx  & =\int_{\Omega} Q_1(x)dx -\int_{\Omega} Q_3(x)dx  \\
   & + \int_{\Omega \cap H_{1,\dots,m-1}}(x_m-p_{m,1})\partial_{x_m} Q_3(x)\big|_{x_m=\pm 1}dx 
  \end{align*}

  and using induction.
  \qed\end{proof}

The recursive subdivision of the problem into sub-problems of lower dimension or degree, as also used in the proof of Theorem \ref{EVN}, can be used to implement a generalization of the classical 
\emph{(inverse) Horner scheme} \cite{Stoer,gautschi,atkinson,endre}. In the case of arbitrary 
multidimensional Newton nodes $P_{m,n}$ with $\tau(P_{m,n}) = \bar P_{m,n}$ for canonical nodes $\bar P_{m,n}$, we can evaluate $Q\circ \tau^{-1}(x)$ with respect to $\bar P_{m,n}$.  
The runtime complexity then increases by adding the cost of inverting the affine transformation $\tau$ from Definition \ref{NTN}, which is given by the cost of inverting an $m \times m$ matrix.

  \section{Approximation Theory}\label{APT} 
Studying how well arbitrary continuous functions can be approximated by polynomials, a fundamental observation was made. The observation is that even though the Weierstrass Theorem states that every function $f : I\subseteq \R \lo \R$, $I=[a,b]$ can be approximated by 
  Bernstein polynomials  in the $C^0$-sense~\cite{weier}, for every choice of interpolation nodes $P_n$ there exists a continuous function for which the interpolant $Q_{n,f}$ does not converge, i.e., $\nol f - Q_{n,f}\nor_{C^0(\Omega)}  \centernot\lo 0$  for $n \rightarrow \infty$~\cite{faber}. This observation is known as {\em Runge's phenomenon}. 
  However, when choosing \emph{Chebyshev nodes} the class of functions that cannot be approximated (i.e., for which Runge's phenomenon occurs) seem to be pathological and extremely unlikely to occur in any ``real world'' application or data set.
We here specify these facts and provide generalized statements for the multidimensional case.

 \subsection{Sobolev Theory}   
We start by providing the analytical setting required to formulate our results.
 In \cite{Sob} an excellent overview of Sobolev theory is given. Here, we simply recall that for every domain $\Omega \subseteq \R^m$ the functional space $C^k(\Omega,\R)$ denotes the real vector space of all functions $f : \Omega \lo \R$ that are $k$ times differentiable in the interior of $\Omega$ 
 and are continuous on the closure $\overline \Omega = \Omega \cup \partial \Omega$. 
 Further, by equipping $C^k(\Omega,\R)$  with the norm 
  $$ \nol f\nor_{C^k(\Omega)} = \sum_{l=0}^k\sup_{x \in \Omega}\sum_{\alpha \in A_{m,l}}\li|\partial^\alpha f(x)\re|\,, \quad \partial^\alpha f(x) = 
\partial^{\alpha_1}_{x_1}\cdots \partial^{\alpha_n}_{x_n}f(x), $$ 
  we obtain a Banach space. We consider the Sobolev space of $L^p$--functions with well-defined weak derivatives up to order $k$, i.e., in general for 
  $1\leq p < \infty$
  \begin{equation}\label{Wkp}
W^{k,p}=\Big\{ f \in L^p(\Omega,\R) \mi \nol f \nor_{W^{k,p}(\Omega)}^p:=\sum_{\alpha\in A_{m,k}}\nol \partial^\alpha f\nor_{L^p(\Omega)}^p < \infty \Big\}\,.
  \end{equation}
The \emph{Hilbert space} $H^k(\Omega,\R):=W^{k,2}(\Omega,\R)$ is of special interest.
In the following, we assume that $\Omega = [-1,1]^m$ is the standard hypercube and denote by $\T^m = \R^m / 2\Z^m$ the torus with fundamental domain $\Omega$.
Then, $H^k(\Omega,\R) \subseteq L^2(\Omega,\R)$, $0 \leq k \leq \infty$, implies that every $f \in H^k(\Omega,\R)$ can be written as a power series 
\begin{equation*} 
  f(x)=   \frac{1}{|\Omega|}\sum_{\alpha \in \N^m} d_\alpha x^\alpha\,, \quad d_\alpha \in \R
\end{equation*}
almost everywhere,  i.e., the identity is violated only on a set $\Omega'\subseteq \Omega$ of Lebesgue measure zero.
Further, for every function  $\widetilde f : \T^m \lo \R$, there is a uniquely determined function $f : \Omega \lo \R$  such that  $\widetilde f(x + 2\Z^m)=f(x)$. Therefore, $f$ can obviously be extended to a periodic function on $\R^m$ and is called the {\em periodic representative} of $\widetilde f$.
Thus, we can write $f$ as a Fourier series 
\begin{equation*}
 f(x)= \frac{1}{|\Omega|}\sum_{\alpha \in \N^m} c_\alpha e^{\pi i \li<\alpha, x\re> }\,, \quad c_\alpha \in \C 
\end{equation*}
and observe that  the space $H^k(\T^m,\R)$ can be defined as
\begin{equation}
  H^k(\T^m,\R)=\overline{C^{\infty}(\T^m,\R)}^{\nol \cdot\nor_{H^k(\Omega)}} ,
\label{Compl}
\end{equation}
which is the completion of $C^{\infty}(\T^m,\R)$ with respect to the $H^k$-norm \cite{Sob}. Therefore, 
\begin{align}\label{frac}
 \nol f \nor_{H^k(\Omega)}^2:= &\sum_{\beta\in A_{k,m}}\li<\partial^\beta f,\partial^\beta f \re>_{L^2(\Omega)} \nonumber \\
 = & \sum_{\beta \in A_{m,k}, \alpha \in \N^m}\big(\pi^{\|\beta\|}\alpha^\beta|c_{\alpha}|\big)^2 \,, 
\end{align}
Thus, $C^{\infty}(\T^m,\R)\subseteq H^k(\T^m,\R)$ is a dense subset and, due to the right-hand side of Eq.~\eqref{frac}, a notion of fractal derivatives can be given, i.e., $H^k(\Omega,\R)$ with $k \in \R^+$ and $\alpha \in \R^m$, $\alpha_i \geq 1$, $|\alpha|\leq k$, can be considered.  
Vice versa, due to the\emph{ Sobolev and Rellich-Kondrachov embedding Theorem} \cite{Sob}, we have that whenever 
$k > m/2$ then $H^k(\T^m,\R) \subseteq C^0(\T^m,\R)$ 
and the embedding 
\begin{equation*} 
i : H^k(\T^m,\R) \hookrightarrow C^0(\T^m,\R) 
\end{equation*}
is well defined, continuous, and compact. Thus, for $m \in \N$ and all periodic representatives $f$ of $ \widetilde f\in \C^0(\T^m,\R)$, there exists a constant $c=c(m,\Omega) \in \R^+$ such that  
\begin{equation}\label{norm}
 \nol f \nor_{C^0(\Omega)} \leq  c\,\nol f \nor_{H^k(\Omega)}
\end{equation}
and for every $B \subseteq H^k(\T^m,\R)$ we have that $i(B) \subseteq C^0(\T^m,\R)$
is precompact whenever $B$ is bounded in $H^k(\T^m,\R)$. By the trace Theorem \cite{Sob}, we observe furthermore that whenever $H \subseteq \R^m$ is a hyperplane of co-dimension $1$, then the induced restriction 
\begin{equation}\label{trace}
  \varrho: H^k(\Omega,\R) \lo H^{k-1/2}(\Omega \cap H,\R) 
\end{equation}
is continuous, i.e., $\nol f_{| \Omega \cap H} \nor_{H^{k-1/2}(\Omega\cap H)} \leq d\nol f \nor_{H^k(\Omega)}$  for some $d=d(m,\Omega) \in \R^+$. 
We consider 
$$\Theta_{m,n}= \Big\{\widetilde f\in H^k(\T^m,\R)\mi  f(x) = \frac{1}{|\Omega|}\sum_{\alpha \in A_{m,n}} c_\alpha e^{\pi i  \li<\alpha, x\re> }\,, c_{\alpha} \in \C\Big\}$$ 
the space of all finite Fourier series of bounded frequencies and denote by 
\begin{align}
 \theta_n : H^k(\T^m,\R)\lo & \Theta_{m,n} \subseteq C^0(\T^m,\R)\,,  \nonumber\\
 \tau_n : H^k(\Omega,\R)\lo & \Pi_{m,n}  \subseteq C^0(\Omega,\R)\label{pro}
\end{align}
the corresponding projections onto $\Theta_{m,n}$ or onto the space of polynomials  $\Pi_{m,n}$. Further, we denote by $\theta_n^\perp = I -\theta_n$, $\tau_n^\perp=I-\tau_n$ the complementary projections. Then, we have:

\begin{lemma} Let $k,m \in \N$, $\Omega =[-1,1]^m \subseteq \R^m$ be the standard hypercube, $\widetilde f \in H^{k}(\T^m,\R)$, and $f$ its periodic representative. Then:  
\begin{enumerate} 
\item[i)]  $\nol f \nor_{C^0(\Omega)} \leq \nol f \nor_{H^k(\Omega)}$ for all $\widetilde f \in H^k(\T^m,\R)$.
 \item[ii)]  $\nol \theta_n^\perp(f)\nor_{C^0(\Omega)} \in o\big((m/n)^{k}\big)$, i.e. 
 $$  (n/m)^{k} \nol \tau_n^\perp(f)\nor_{C^0(\Omega)} \xrightarrow[n\rightarrow \infty]{} 0\quad \text{for every}\,\, m \in \N\,.$$
 \item[iii)] For $n,l \in \N$, the operator norm of  $\tau_{n}^\perp \theta_l: H^k(\T^m,\R) \lo C^0(\T^m,\R)$ is bounded by $\nol \tau_{n}^\perp \theta_l \nor \leq N(m,l)^{1/2}\frac{(\pi l^2)^{n+1}}{(n+1)!}$.
\end{enumerate}
\label{shrink}
 \end{lemma}
 
\begin{proof} To show $(i)$, we approximate $f$ by a finite Fourier series, i.e., we assume that  $f(x) = \frac{1}{|\Omega|}\sum_{\alpha \in A_{m,n}}c_\alpha e^{\pi i  \li<\alpha,x\re>}$. 
The Fourier basis is orthonormal, i.e., 
$$\frac{1}{|\Omega|}\int_\Omega  e^{\pi i  \li<\alpha,x\re>}\cdot e^{-\pi i  \li<\beta,x\re>} \mathrm{d}\Omega =  \li\{ \begin{array}{ll}
                                                                                                           1 \,, & \text{if} \,\,\, \alpha =\beta\\
                                                                                                           0\, & \text{else.}
                                                                                                          \end{array}\re.
$$
Thus,  we compute
\begin{align*}
 \nol f \nor^2_{C^0(\Omega)} =&\sup_{x \in \Omega}\Big|\sum_{\alpha \in A_{m,n}}c_\alpha e^{\pi i  \li<\alpha,x\re>}\Big|^2 \leq \sum_{\alpha \in A_{m,n}}|c_\alpha|^2 
  =\nol f \nor^2_{L^2(\Omega)}  \leq    \nol f \nor_{H^k(\Omega)}^2 \,.
\end{align*}
Now $(i)$ follows from the density of the approximation given in Eq.~\eqref{Compl} and the continuity of the norm $\nol \cdot \nor_{H^k(\Omega)}$. 
To show $(ii)$, let $\alpha \in \N^m$ 
and observe that 
$$\partial^\beta(c_\alpha e^{\pi i \li<\alpha,x\re>})=(i\pi)^{\|\beta||}\alpha^{\beta} e^{\pi i \li<\alpha,x\re>} \,, \,\,\,\beta\in A_{m,k}\,.$$
If $|\alpha|>m $ then at least one $\alpha_{i} > n/m$  for some $1 \leq i \leq m$. Thus, by choosing $\beta \in A_{m,k}$ with $\beta_{i}=k$, we obtain  
$$\big|\partial^\beta(c_\alpha e^{\pi i \li<\alpha,x\re>} )\big| \geq (\pi n/m)^{k}|c_\alpha|\,, \quad \text{for} \,\,\,|\alpha|\geq m\,.$$
Hence, for $n >m$ 
\begin{align*} 
(n/m)^{2k} ||\theta_n^\perp f||_{C^0(\Omega)}^2  \leq  ||\theta_n^\perp f||_{H^k(\Omega)}^2 \,.
\end{align*}
Due to Eq.~\eqref{Wkp} we have  $||\theta_n^\perp f||_{H^k(\Omega)} \xrightarrow[n\rightarrow \infty]{} 0$, which yields  $||\theta_n^\perp f||_{C^0(\Omega)} \in o\big((m/n)^{k}\big)$ for every fixed $m \in \N$, proving $(ii)$.
To show $(iii)$, we assume  $\nol f\nor_{H^k(\Omega)} \leq 1$, write 
$$\theta_l(f) = \sum_{\alpha \in A_{m,l}} c_{\alpha}e^{\pi i \li<\alpha,x\re>}\,, c_\alpha \in \C , $$ 
and  recall that for every $x \in \R^m$ there exists $\xi_x\in \Omega$ such that 
$$e^{i\pi\li<\alpha,x\re>} = \sum_{n \in \N}\frac{(i\pi\li<\alpha,x\re>)^n}{n!} =  \sum_{h \leq n }\frac{(i\pi\li<\alpha,x\re>)^h}{h!} + \frac{\partial^{n+1}_ve^{i\pi\li<\alpha,\xi_x\re>}}{(n+1)!}(i\pi\li<\alpha,x\re>)^{n+1}\,,$$ where 
$\partial_v$ denotes the partial derivative in direction $v = \frac{x}{|x|} \in \R^m$. Therefore, the last term is the \emph{Lagrange remainder}. Now 
$$ \partial^{n+1}_ve^{i\pi\li<\alpha,\xi_x\re>}= \partial^{n}_v\big<\nabla e^{i\pi\li<\alpha,\xi_x\re>},v\big> = \big<\alpha,v\big>\partial^{n}_ve^{i\pi\li<\alpha,\xi_x\re>}= \big<\alpha,v\big>^{n+1}e^{i\pi\li<\alpha,\xi_x\re>}$$ 
and $|\li<\alpha,v\re>|, |\li<\alpha,x\re>| \leq |\alpha|$ imply that 
$$\li|\frac{\partial^{n+1}_ve^{i\pi\li<\alpha,\xi_x\re>}}{(n+1)!}(i\pi\li<\alpha,x\re>)^{n+1}\re| \leq \frac{ (\pi|\alpha|)^{2n+2}}{(n+1)!} . $$ 
Hence, due to $|\alpha|\leq l$, $|c_\alpha| \leq 1$, we bound
\begin{align*}
 \nol\tau_{n}^\perp \theta_l(f)\nor_{C^0(\Omega)} & \leq  \li(\sum_{\alpha \in A_{m,l}} |c_{\alpha}|^2 \li(\frac{ (\pi|\alpha|)^{2n+2}}{(n+1)!}\re)^2\re)^{1/2}\\
 &\leq 
                                            \frac{\pi^{n+1}}{(n+1)!}\li(\sum_{\alpha \in A_{m,l}}|\alpha|^{4(n+1)}\re)^{1/2}\\
                                            &\leq  \frac{\pi^{n+1} l^{2n+2}}{(n+1)!}N(m,l)^{1/2} ,
 \end{align*}
which yields $(iii)$.
\qed\end{proof}

\subsection{Lebesgues Functions} 
Lebesgue functions measure the relative approximation error of an interpolation scheme. More precisely: 
Let $m,n \in \N$, $\Omega= [-1,1]^m \subseteq \R^m$, and $S_{m,n}:C^0(\Omega,\R) \lo \Pi_{m,n}$ denote an interpolation scheme with respect to unisolvent interpolation nodes $P_{m,n}$. 
That is, for $f \in C^0(\Omega,\R)$, $S_{m,n}(f) \in \Pi_{m,n}$ solves the PIP with respect to $(m,n,f,P_{m,n})$. Since $P_{m,n}$ is unisolvent and independent of $f$,  
it is readily verified that $S_{m,n}$ is a linear operator. Therefore, the following is well defined:

\begin{definition}[Operator norms and Lebesgues functions of interpolation operators] 
Let $m,n \in \N$, $\Omega=[-1,1]^m$, and $P_{m,n}\subseteq \Omega$ be unisolvent interpolation nodes. Consider the interpolation operator 
$S_{m,n}:C^0(\Omega , \R) \lo \Pi_{m,n}$ and its restriction $S_{m,n | H^k(\Omega)} : H^k(\Omega,\R) \lo \Pi_{m,n}$. Then we define by 
\begin{align*}
 \Lambda(P_{m,n},C^0(\Omega)): = &  \sup_{ f\in C^0(\Omega,\R), ||f||_{C^0(\Omega)} =1} ||S_{m,n}(f)||_{C^0(\Omega)} \\ 
  \Lambda(P_{m,n},H^k(\Omega)): = &  \sup_{ f\in H^k(\Omega,\R), ||f||_{H^k(\Omega)} =1} ||S_{m,n | H^k(\Omega)}(f)||_{C^0(\Omega)}
\end{align*}
the {\em operator norms} of $S_{m,n}$ and $S_{m,n | H^k(\Omega)}$, respectively. Denote with $\Pc_{m,n}$ the set of all unisolvent node sets with respect to $m,n \in \N$ then 
\begin{align*}
  \Lambda_{m,n,C^0(\Omega)} : \Pc_{m,n}\lo \R \,, &  \quad P_{m,n} \mapsto  \Lambda(P_{m,n},C^0(\Omega))\,,\\ 
  \Lambda_{m,n,H^k(\Omega)} : \Pc_{m,n}\lo \R \,, &  \quad P_{m,n} \mapsto   \Lambda(P_{m,n},H^k(\Omega))  
\end{align*}
are called the \emph{Lebesgue functions}  with respect to the considered norms.
\end{definition}
Due to Lemma \ref{shrink},  
$$\li\{f\in H^k(\Omega,\R) \mi  ||f||_{H^k(\Omega)} =1\re\} \subseteq \li\{f\in C^0(\Omega,\R) \mi ||f||_{C^0(\Omega)} =1\re\}, $$ implying that 
\begin{equation}\label{LEB}
 \Lambda(P_{m,n},H^k(\Omega))   \leq \Lambda(P_{m,n},C^0(\Omega)) \,.
\end{equation}

Understanding the behavior of Lebesgue functions in 1D is crucial for extending their definition to arbitrary dimensions. 
In particular, the following observation is key to our further considerations:
\begin{lemma} \label{PartLEB} Let $l,n \in \N$, $l,n \geq 1$, $\Omega=[-1,1]^m$, $P_n= \{p_1,\dots,p_{n+1}\}$  be a set of \mbox{$n+1$} pairwise disjoint nodes, and $P_l= \{p_1,\dots,p_{l+1}\} \subseteq P_n$.  
Then 
$$\Lambda(P_l,C^0(\Omega)) \leq \Lambda(P_n,C^0(\Omega))\,, \quad \Lambda(P_l,H^k(\Omega)) \leq \Lambda(P_n,H^k(\Omega)) \text{ for } k > m/2. $$
\end{lemma}
\begin{proof} We choose small intervals $I_{h,\ee}=[p_{h}-\ee,p_{h}+\ee]$, $I_{h,\delta}=[p_{h}-\delta,p_{h}+\delta]$, $h >l+1$ with $\ee> \delta >0$ small enough so that $I_{h,\ee}\cap P_n = \{p_h\}$. Further,
we consider smooth cut--off functions 
$\beta_h : [p_h -\ee,p_h +  \ee] \lo [0,1]$,  with  $\beta_h(p_h\pm \ee)=1$, $\beta_h^\pm(p_h\pm \delta)=0$. We denote by $S_{1,k}$, $S_{1,n}$ the interpolation schemes with respect to $P_k$, $P_n$, respectively, and 
set  
$$ \tilde f(x)= \li\{\begin{array}{ll}
                      f(x) \,, & \text{if} \,\,\, x \not \in I_{h,\ee}, \,\,\, h>k+1 \\
                      \beta_h^\pm(x)f(x) + (1- \beta_h^\pm(x))f(p_{l+1})\,, & \text{if} \,\,\, x \in [p_h\pm\delta,p_h\pm\ee] \,, \,\,\, h>l+1 \\
                      f(p_{l+1})\,, & \text{if} \,\,\, x \in I_{h,\delta}  \,\,\, h>k+1 .
                      \end{array}\re.\,
$$ 
Obviously, it is $\nol \tilde f\nor_{C^0(\Omega)}, \nol \tilde f\nor_{H^k(\Omega)}\leq 1$. Moreover, following the alternative divided difference scheme from Eq.~\eqref{DDS2}, we find $f_j(p_j)=0$ for all  $j>l+1$.
Thus, the coefficients $c_j$, $j >l$, of the polynomial $S_{1,n}(\tilde f)$ in Newton form vanish, while the 
$c_0,\dots,c_{l}$ are given by $\tilde f_0(p_1) = f_0(p_1),\dots, \tilde f_l(p_{l+1}) =f_l(p_{l+1})$. Hence, 
$S_{1,n}(\tilde f) =S_{1,l}(f) = S_{1,l}(\tilde f)$. Therefore, 
$$\nol S_{1,n}(\tilde f)\nor_{C^0(\Omega)}=\nol S_{1,l}(\tilde f)\nor_{C^0(\Omega)} =\nol S_{1,l}(f)\nor_{C^0(\Omega)} \,.$$
Observing that  $f$ was arbitrarily chosen, this completes the proof. 
\qed\end{proof}

\subsection{Newton-Chebyshev Nodes}
In order to extend the study of Runge's phenomenon to multiple dimensions, we introduce a multidimensional notion of \emph{Chebyshev nodes} and provide the essential 
approximation results.
\begin{definition}[Multidimensional Newton-Chebyshev nodes] Let $m,n \in \N$, $T_{m,n}$, and $\Gamma_{m,n}$ be given. Let $P_{m,n}$ be the canonical multidimensional Newton nodes generated by  
$$\PP_{m,n}= \oplus_{l=1}^m \Cheb_n\,,$$
where $\Cheb_n$ was defined in Eq.~\eqref{cheby}.
Then, we call $P_{m,n}$ {\em canonical multidimensional Newton-Chebyshev nodes}, and we call every affine transformation $\tau(P_{m,n})$ of $P_{m,n}$ {\em multidimensional Newton-Chebyshev nodes}. 
\end{definition}

\begin{theorem} Let $m,n,k \in \N$, $k >m/2$, and  $S_{m,n}$ be an interpolation operator with respect to multidimensional Newton nodes $P_{m,n}$ generated by $\PP_{m,n}= \oplus_{l=1}^m P_l$, $\#P_l=n+1$. Then 
$$ \Lambda(P_{m,n},H^k(\Omega)) \leq \prod_{l=1}^m \Lambda(P_l, H^{k-(m-1)/2}([-1,1]))\,.$$
If $P_{m,n}$ are multidimensional Newton-Chebyshev nodes, then in particular 
\begin{equation}\label{MC}
 \Lambda(P_{m,n},H^k(\Omega)) \leq \Lambda(\Cheb_n,C^0([-1,1]))^m \in \Oc(\log(n)^m) \,.
\end{equation}
\label{NCL}
\end{theorem}

\begin{proof} We argue by induction on $m$. For $m=1$ the claim directly follows from Eq.~\eqref{LEB}. If $m>1$ then we can assume w.l.o.g.~(i.e., by changing coordinates if necessary) that $P_{m,n}$ are canonical Newton-Chebyshev nodes, and we can consider the hyperplanes $H_1,\dots,H_{n}$ given by 
$H_i = \{x \in \R^m \mi x_m =p_{m,i} \in P_{m}\}$, $i =1,\dots,n$. Denote by $Q_{H_i}(x)=x_m -p_{m,i}$ the linear polynomial defining $H_i = Q_{H_i}^{-1}(0)$ and by $\pi_{H_i}: \R^m \lo H_i$, $ \pi_{H_i}(x) = (x_1,\dots,x_{m-1},p_{m,i})$, the corresponding projections. Then, by Theorems \ref{GN}, \ref{SV}, and Proposition \ref{UniN}, we have
\begin{align}
  S_{m,n}(f) =&  S_{m-1,n,H_1}(f) + Q_{H_1}\Big( S_{m-1,n-1}(f_1) + Q_{H_2}\big( S_{m-1,n-2}(f_2) + \cdots  \nonumber\\ 
  &Q_{H_n}\big(S_{m-1,1}(f_{n-1}) + S_{m,0}(f_n) \big) \dots \Big)\,,  \label{NI}
\end{align}
where 
$$f_0 =f\,, \quad f_k = \frac{f_{k-1}(x)-f_{k-1}(\pi_{H_k}(x))}{Q_H(x)} $$
and the $S_{m-1,n-i}$ interpolate with respect to $P_{m-1,n-i}$ generated by 
$$\PP_{m-1,n-i} = \{p_{m,i+1}\}\times \oplus_{l=1}^{m-1} P_l \,.$$
By observing that the interpolant $S_{m-1,n-i}(f_i)$ is constant along directions normal to $H_i$, and by identifying $H_i \cong \R^{m-1}$, induction and Lemma \ref{PartLEB} yield 
$$\Lambda(P_{m-1,n-i},H^k(\Omega)) \leq \prod_{l=1}^{m-1} \Lambda(P_l,H^{k-(m-2)/2}([-1,1]^{m-1})) =:\Lambda^*\,.$$
Though the $f_i$ are multivariate functions, the remaining interpolation is done with respect to $x_m$ in only $1$ variable. Hence, recalling the 1D estimation, we get
\begin{align*}
  \nol S_{m,n}(f)  \nor_{C^0(\Omega)} & \leq  \Lambda^*\nol f_0 + Q_{H_1}\big( f_1  + \cdots  +Q_{H_n}\big(f_{n-1} + f_n \big) \dots \big)\nor_{C^0(\Omega)}\\
                                      & \leq \Lambda^*\Lambda(P_m,H^{k-(m-1)/2}([-1,1]))\\
                                      &\leq \prod_{l=1}^{m} \Lambda(P_l,H^{k-(m-1)/2}([-1,1]))\,.
\end{align*}
This proves the first statement. The proof of Eq.~\eqref{MC} follows from Theorem \ref{Cheb}. 
\qed\end{proof}

\begin{theorem} Let $m,n,k \in \N$, $k>m/2$, $\Omega= [-1,1]^m \subseteq \R^m$, $\T^m = \R^m/2\Z^m$, and $\widetilde f \in H^k(\T^m,\R)$ with periodic representative $f$. 
Consider  an interpolation operator $S_{m,n}:H^k(\Omega,\R) \lo \Pi_{m,n}$ with respect to multidimensional Newton-Chebyshev nodes $P_{m,n}$.
Then  
  $$\nol f -S_{m,n}(f)\nor_{C^0(\Omega)} \xrightarrow[n\rightarrow \infty]{}0 \,. $$
\end{theorem}

\begin{proof} The proof is based on balancing the approximation of $f$ in Fourier basis and by polynomials. To do so, we let $n \in \N$, $D,K_1,K_0 \in \R^+ $, $ D\geq 1$, $K_1>K_0\geq 2 $, and $l=l(n) \in \N$ such that 
\begin{equation}\label{bal}
  \frac{(Dl)^{K_1}}{n}  \xrightarrow[n\rightarrow \infty]{} 0  \quad \text{and} \quad  \frac{\log(n)^{K_0}}{l}  \xrightarrow[n\rightarrow \infty]{} 0 \,,
\end{equation}
i.e., $(Dl)^{K_1} \in o(n)$ and $\log(n)^{K_0}\in o(l)$. 
We consider the projections $\tau_n,\tau_n^\perp,\theta_l,\theta_l^\perp$ from Eq.~\eqref{pro} and use the fact that $S_{m,n}$ is a linear operator in order to bound
\begin{align*}
  \nol f -S_{m,n}(f)\nor_{C^0(\Omega)}  & \leq  \nol \theta_l(f) -S_{m,n}(\theta_l(f))\nor_{C^0(\Omega)}  \\
  &+ \nol \theta_l^\perp(f) -S_{m,n}(\theta_l^\perp(f))\nor_{C^0(\Omega)} \,.
\end{align*}
We then use Theorem \ref{NCL} and Lemma \ref{shrink}$(i)$ to  observe that there exists a constant $c\in \R^+$ so that the second term on the right-hand side satisfies 
\begin{align*}
 \nol \theta_l^\perp(f) -S_{m,n}(\theta_l^\perp(f))\nor_{C^0(\Omega)}  & \leq  \nol \theta_l^\perp(f) -S_{m,n}(\theta_l^\perp(f))\nor_{H^k(\Omega)} \\
 &\leq \big(1+\Lambda(S_{m,n},H^k(\Omega))\big) \nol \theta_l^\perp(f)\nor_{H^k(\Omega)} \\
 									&\leq c\big(1+\log(n)^m)\big)\nol \theta_l^\perp(f)\nor_{C^0(\Omega)} \,.
\end{align*}
By  Lemma \ref{shrink}$ii)$, we find that $\nol \theta_l^\perp(f)\nor_{C^0(\Omega)}  \in  o((m/l)^k)$ for $k>m/2$. Since $K_0 \geq 2$, we obtain     
$$  \nol \theta_l^\perp(f) -S_{m,n}(\theta_l^\perp(f))\nor_{C^0(\Omega)}  \leq \frac{1+\log(n)^m}{l^{m/2}}\nol f\nor_{C^0(\Omega)}  \xrightarrow[n\rightarrow \infty]{} 0 \,.$$ 
Similarly, we bound the remaining term by 
\begin{align*}
\nol \theta_l(f) -S_{m,n}(\theta_l(f))\nor_{C^0(\Omega)}  & \leq  \nol \pi_n\theta_l(f) -S_{m,n}(\pi_n\theta_l(f))\nor_{C^0(\Omega)} \\
&+ \nol \pi_n^\perp\theta_l(f) -S_{m,n}(\pi_n^\perp\theta_l(f))\nor_{C^0(\Omega)}. 
\end{align*}
Since $\pi_n\theta_l(f) \in \Pi_{m,n}$, we have $S_{m,n}(\pi_n\theta_l(f))=\pi_n\theta_l(f)$, implying that the first term vanishes. 
Now we again use Theorem \ref{NCL} and Lemma \ref{shrink}$(i)$ and \ref{shrink}$(iii)$ to observe that there exists a constant $c\in \R^+$ so that 
\begin{align*} 
 \nol \pi_n^\perp\theta_l(f) -S_{m,n}(\pi_n^\perp\theta_l(f))\nor_{C^0(\Omega)} & \leq  \nol \pi_n^\perp\theta_l(f) -S_{m,n}(\pi_n^\perp\theta_l(f))\nor_{H^k(\Omega)} \\
 & \leq c(1 + \log(n)^m)\nol \pi_n^\perp\theta_l(f)\nor_{H^k(\Omega)}\\ 
 & \leq c(1 + \log(n)^m)N(m,l)^{1/2}\frac{(\pi l^2)^{n+1}}{(n+1)!} \nol f\nor_{C^0(\Omega)}. 
\end{align*}
Since $N(m,l)=\frac{(m+l)!}{n!l!} = \frac{(m+1)\cdots(m+l)}{l!} \in \Oc(l^m)$ and because $n! >(n/2)^{n/2}$ for all $n \in \N$, there exists a contant $d\in \R^+$ so that  
\begin{align*} 
c(1 + \log(n)^m)N(m,l)^{1/2}\frac{(\pi l^2)^{n+1}}{(n+1)!}   &  \leq d(1+l^m)l^{m/2}\li(\frac{(\pi l^2)}{n+1} \re)^{\!\! n+1} .
\end{align*}
Choosing $D, K_1 \in \R^+$ large enough we can further use Eq.~\eqref{bal} to bound
\begin{align*}
  \nol \pi_n^\perp\theta_l(f) -S_{m,n}(\pi_n^\perp\theta_l(f))\nor_{C^0(\Omega)} &  \leq d(1+l^m)l^{m/2}\frac{(\pi l^2)^{n+1}}{n^{(n+1)/2}}   \nol f \nor_{C^0(\Omega)}\\
  & \leq \li(\frac{(Dl)^{K_1}}{n+1}\re)^{\!\!(n+1)/2}  \nol f \nor_{C^0(\Omega)}\xrightarrow[n\rightarrow \infty]{} 0 \,.
\end{align*}
Hence, all terms converge to zero as $n \rightarrow \infty$, proving the theorem.
\qed\end{proof}

\subsection{Approximation Errors }
We generalize the classic estimates of approximation errors in 1D to arbitrary dimensions $m \in \N$. 

\begin{theorem} Let $m,n \in \N$ and $\Omega= [-1,1]^m \subseteq \R^m$. 
 Let $S_{m,n}:H^k(\Omega,\R) \lo \Pi_{m,n}$, $k>m/2$, denote an interpolation operator with respect to canonical multidimensional Newton nodes $P_{m,n}$ generated by $\PP_{m,n}=\oplus_{i=1}^m P_i$, $P_i=\li\{p_{i,1},\dots,p_{i,n+1}\re\}$.
    \begin{enumerate}
    \item[i)] If $f \in C^{n+1}(\Omega,\R)$  then for every $\alpha \in A_{m,n}\setminus A_{m,n-1}$, $i \in\{1,\dots,m\}$  and every $x \in \Omega$ there is 
  $\xi_x\in \Omega$ such that 
\begin{equation}\label{AE}
   | f(x)- S_{m,n}(f)(x) |\leq \frac{1}{\alpha_i!} \partial^{\alpha_i+1}_{x_i}f(\xi_x) |N_\alpha(x)|\,, 
\end{equation}
where  $N_\alpha(x) = \prod_{i=1}^m\prod_{j=1}^{\alpha_i}(x_i-p_{i,j})$, $x =(x_1,\dots,x_m)$, $p_{i,j}\in P_i$.
If  $P_{m,n}$ are multidimensional Newton-Chebyshev nodes, then we can further bound 
\begin{equation}\label{AEE} 
| f(x)- S_{m,n}(f)(x) |  \leq \frac{1}{2^{\alpha_i}\alpha_i! }\partial^{\alpha_i+1}_{x_i} f(\xi_x)\,.
\end{equation}
  \item[ii)] 
 For any unisolvent node set $P_{m,n}$ the relative interpolation error is
  $$\nol f- S_{m,n}(f) \nor_{C^0(\Omega)} \leq (1+\Lambda(P_{m,n},H^k(\Omega))\nol f- Q^*_{m,n} \nor_{C^0(\Omega)}$$
  for all $f \in H^k(\Omega,\R)$, where  $Q^*_{m,n}$ is an optimal approximation that minimizes the $C^0$-distance to $f$.
 \end{enumerate}
\end{theorem}
\begin{proof} Changing coordinates if necessary, we can assume w.l.o.g.~that $P_{m,n}$ are canonical nodes. To show $(i)$, we follow the argumentation of the classical proof in 1D~\cite{gautschi}.
We consider the line 
$$L_{\alpha, i}= \li\{x \in \R^m \mi x_j = p_{j,1+\alpha_j}\,\,\text{for}\,\, 1 \leq j< i \,, \quad  x_j = p_{j,1+\alpha_j}  \,\,\text{for}\,\, i<j\leq m  \re\}$$
and choose $\bar x \in (\Omega \cap L_{\alpha,i})\setminus P_{m,n}$.  The function $g_{\alpha,i} : \R^m \lo \R$ given by 
$$g_{\alpha,i}(x) = f(x) -Q_{m,n,f}(x) - G(\bar x) N_\alpha(x) \,, \,\,\, G(\bar x)= \big(f(\bar x) -Q_{m,n,f}(\bar x)\big)/N_\alpha(\bar x) $$
is of class $C^{n+1}$ and possesses \mbox{$\alpha_i +1$} roots, namely $\{p_{i,1},\dots,p_{i,\alpha_i}, \bar x\}$, on $L_{\alpha, i}$. 
Recursively applying Rolle's Theorem, this implies that  $\partial^{k}_{x_i}g_{\alpha,i}$ possesses \mbox{$\alpha_i+1-k$} roots on $L_{\alpha,i}$, $0 \leq k \leq \alpha_i$. Hence 
$$ \partial^{\alpha_i+1}_{x_i} g_\alpha(\xi)= 0 \,, \quad \text{for some}\,\,\, \xi=\xi_{\bar x} \in L_{\alpha,i} \cap \Omega \,.$$ 
Theorems \ref{GN} and \ref{SV} yield Eq.~\eqref{NI}, which implies that the restriction of $Q_{m,n,f}$ to $L_{\alpha,i}$  is of degree $\alpha_i$. Thus, $\partial^{\alpha_i +1}_{x_i} Q_{m,n,f\,| L_{\alpha,i}}= 0$. 
Since $\partial^{\alpha_i}_{x_i} N_\alpha(x) = \alpha_i!$, this yields  
$$G(\bar x) = \frac{1}{\alpha_i! }\partial^{\alpha_i+1}_{x_i} f(\xi), $$ 
implying Eq.~\eqref{AE}. Combining Lemma \ref{PartLEB} with the classic error estimation for Chebyshev nodes in 1D \cite{gautschi} yields Eq.~\eqref{AEE}. To show $(ii)$, we recall that $S_{m,n}$ is a projection, i.e., 
$S_{m,n}(S_{m,n})=S_{m,n}$ implying that $S_{m,n}(Q) = Q$ for any $Q \in \Pi_{m,n}$. Thus, for every $f \in H^k(\Omega,\R)$, we bound
\begin{align*}
 \nol f- S_{m,n}(f) \nor_{C^0(\Omega)}  & = \nol f- Q_{m,n}^* + Q_{m,n}^* -  S_{m,n}(f) \nor_{C^0(\Omega)}  \\ 
					& \leq \nol f- Q_{m,n}^*\nor_{C^0(\Omega)}  +  \nol Q_{m,n}^* -  S_{m,n}(f) \nor_{C^0(\Omega)}  \\ 
					& = \nol f- Q_{m,n}^*\nor_{C^0(\Omega)}  +  \nol S_{m,n}(Q_{m,n}^*) -  S_{m,n}(f) \nor_{C^0(\Omega)}  \\ 
					& \leq (1+\Lambda(S_{m,n}),H^k(\Omega))\nol f- Q_{m,n}^*\nor_{C^0(\Omega)} .
\end{align*}
\qed\end{proof}

In summary, we have established all approximation results of Main Result III as stated in Theorem \ref{III} in Section \ref{main}, thereby extending the well-known classic results from 
1D to arbitrary dimensions. This answers Question \ref{Task} set out in the problem statement.

 \section{Numerical Experiments} 
 \label{EX}
In order to illustrate our approach and demonstrate its performance in practice, we implement a prototype MATLAB (R2018a (9.4.0.813 654) version of the PIP-SOLVER running on an Apple MacBook Pro (Retina, 15-inch, Mid 2015) with a 2.2\,GHz Intel Core i7 processor and 16\,GB 1600\,MHz DDR3 memory unser macOS Sierra (version 10.12.6.).
The following numerical experiments demonstrate the computational performance and approximation accuracy of our solver in comparison with the classic numerical approaches.

\begin{figure}[t!]
\begin{minipage}[t]{0.5\textwidth}
\vspace{-2.0cm}
  \includegraphics[scale=0.3]{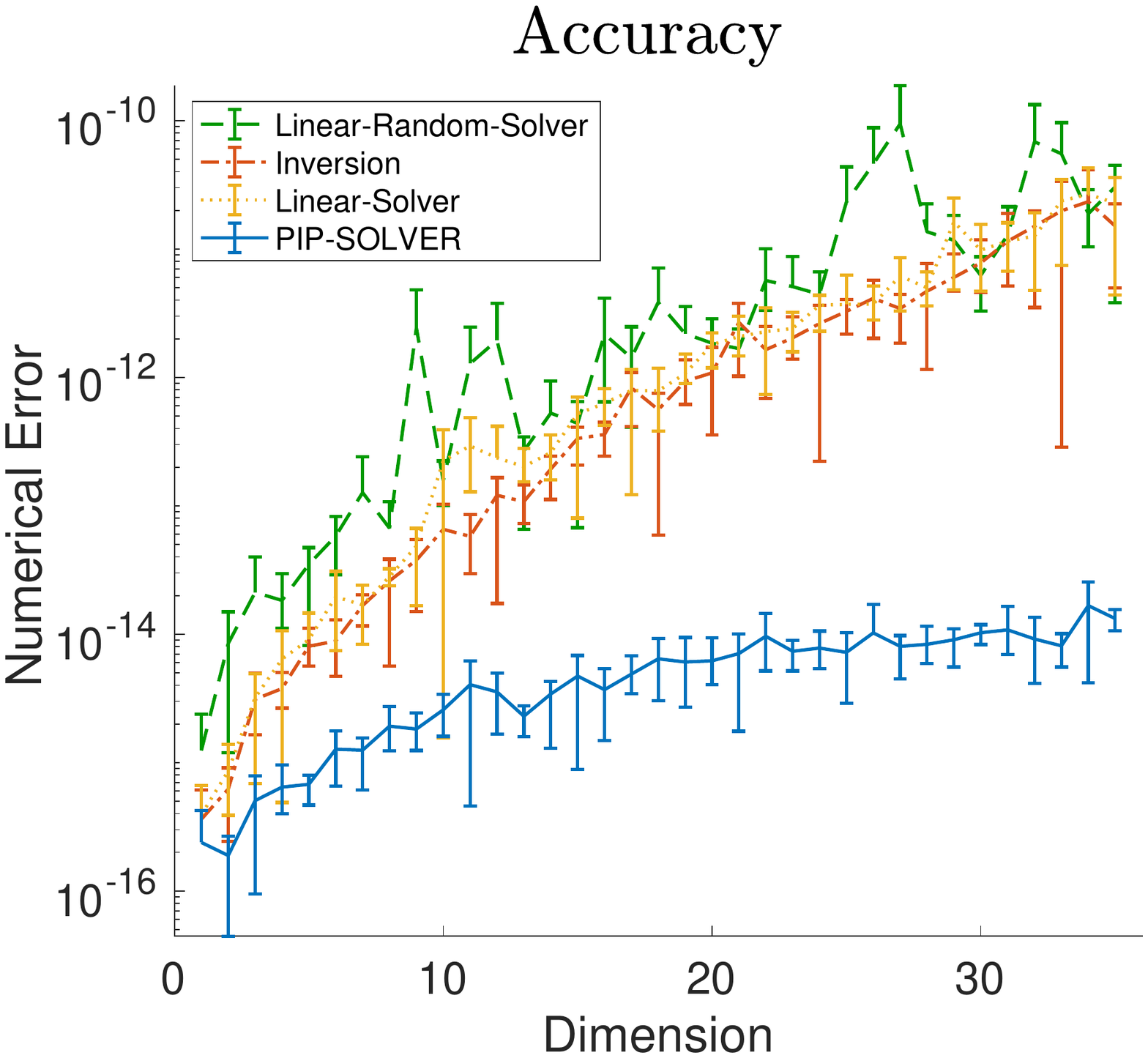}
  \vspace{-2.5cm}
  \caption{Numerical error for degree $n=3$.} \label{ACC1}
 \end{minipage}
\begin{minipage}[t]{0.5\textwidth}
\vspace{-2.1cm}
  \includegraphics[scale=0.3]{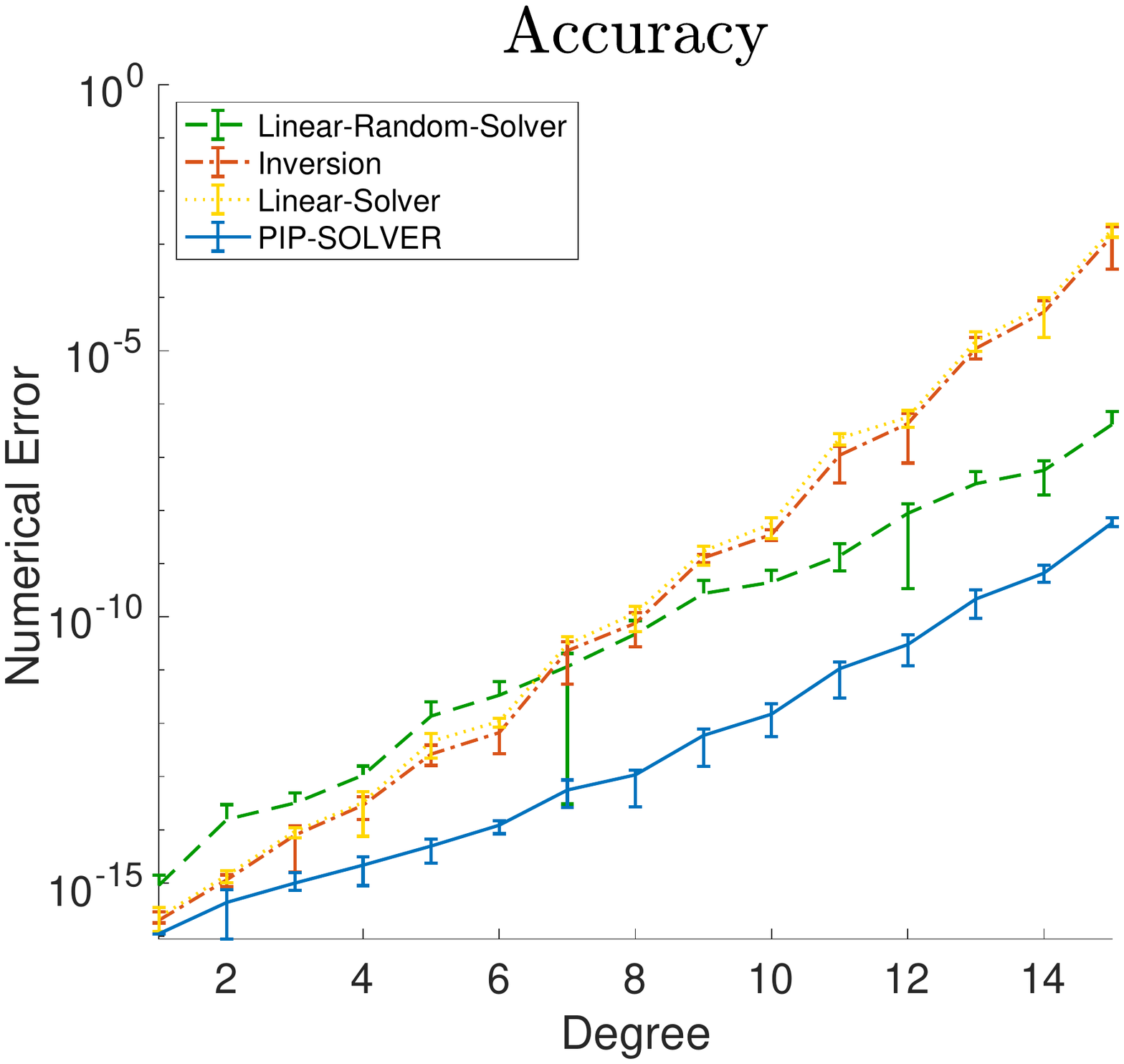}
  \vspace{-2.4cm}
  \caption{Numerical error for dimension $m=5$.} \label{ACC2}
 \end{minipage}
\end{figure}

 For given $m,n \in \N$ and function $f: \R^m \lo \R$, we compare  the following methods in terms of accuracy and runtime: 
 \begin{enumerate}
  \item [i)] The \emph{PIP-SOLVER} generates multidimensional Newton-Chebyshev nodes and determines the coefficients $C \in \R^{N(m,n)}$ of the interpolant $Q_{m,n,f}$ in multivariate Newton form.  
 \item[ii)] The \emph{Linear Solver} uses the multidimensional Newton-Chebyshev nodes $P_{m,n}$ generated by the PIP-SOLVER and then applies the MATLAB linear system solver to solve 
 $$V_{m,n}(P_{m,n})C=F\,, \,\, F=\big(f(p_1),\dots,f(p_N(m,n))\big)^\mathsf{T} \in \R^{N(m,n)}$$ for the coefficients $C \in \R^{N(m,n)}$ of the interpolant $Q_{m,n,f}$ in normal form.
  \item[iii)] The \emph{Linear Random Solver} uses nodes $P_{m,n}$ placed uniformly at random and then applies the MATLAB linear system solver to solve
 $$V_{m,n}(P_{m,n})C=F\,, \,\, F=\big(f(p_1),\dots,f(p_N(m,n))\big)^\mathsf{T} \in \R^{N(m,n)}$$ for the coefficients $C \in \R^{N(m,n)}$ of the interpolant $Q_{m,n,f}$ in normal form. 
 \item[iv)] The \emph{Inversion} method uses the multidimensional Newton-Chebyshev nodes $P_{m,n}$ generated by the PIP-SOLVER and then inverts the Vandermonde matrix $V_{m,n}(P_{m,n})$ using LU-decomposition to compute the coefficients $C \in \R^{N(m,n)}$ of the interpolant $Q_{m,n,f}$ in normal form. 
  \end{enumerate}

\begin{experiment}
We first compare the accuracy of the three approaches, which also serves to validate our method. 
To do so, we choose uniformly-distributed random numbers $c_1,\dots,c_{N-1} \in [-1,1]^{N(m,n)}$  
to be the  coefficients of a polynomial $Q \in \Pi_{m,n}$ in normal or multivariate Newton form. 
We then set $f = Q$ and measure the maximum absolute error in any coefficient, i.e.~$\nol c- \widetilde{c} \nor_\infty$, when recovering $\widetilde Q$ by solving the PIP with respect to $(m,n,f)$. 
 \end{experiment}
 
Figures \ref{ACC1} and \ref{ACC2} show the average and min-max span of the numerical errors (over 5 repetitions with different $i.i.d.$~random coefficients; same 5 for every approach) 
for fixed degree $n=3$ and dimensions $m=2,\dots,35$, as well as for fixed dimension $m=5$ and degree $n=1,\dots,15$, with logarithmic scale in the $y$-axis.

The case $n=3$ is of high practical relevance, e.g., when interpolating cubic splines.
In both cases, all methods show high accuracy, which reflects the fact that Newton-Chebyshev nodes yield well-conditioned PIPs. 
This is confirmed by the \emph{Linear Solver} and the \emph{Inversion} method showing comparable accuracy, while the \emph{Linear Random Solver} is less accurate. 
The error of the PIP-SOLVER is almost constant on the level of the machine accuracy (double-precision floating-point number types).
Especially in high dimensions, the PIP-SOLVER is several orders of magnitude more accurate than the other approaches.  

\begin{figure}[t!]
\begin{minipage}[t]{0.5\textwidth}
\vspace{-2.2cm}
  \includegraphics[scale=0.34]{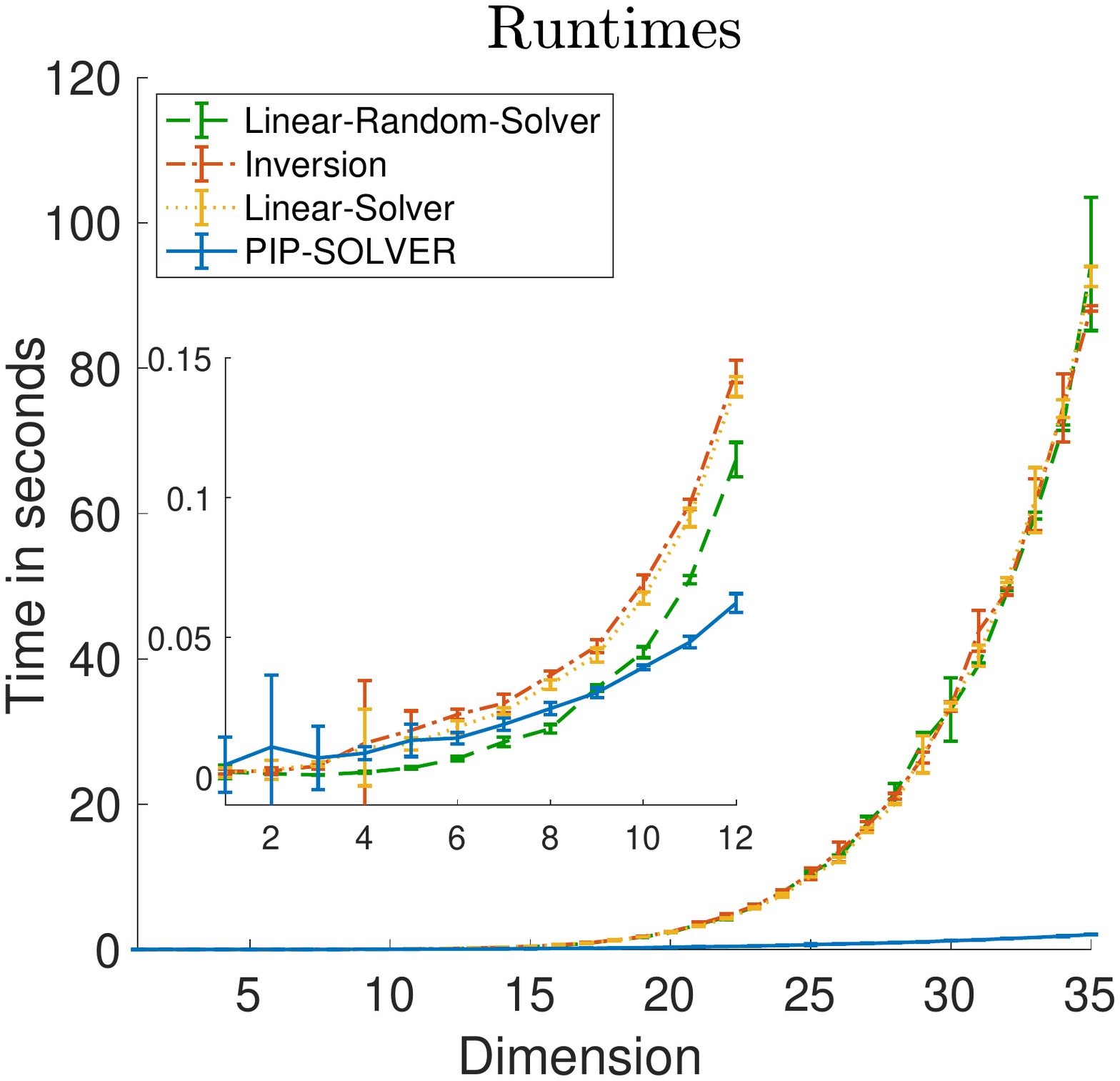}
  \vspace{-2.9cm}
  \caption{Runtimes for $n=3$, $1\leq m\leq 35$.} \label{time1}
 \end{minipage}
\begin{minipage}[t]{0.5\textwidth}
\vspace{-2.25cm}
  \includegraphics[scale=0.34]{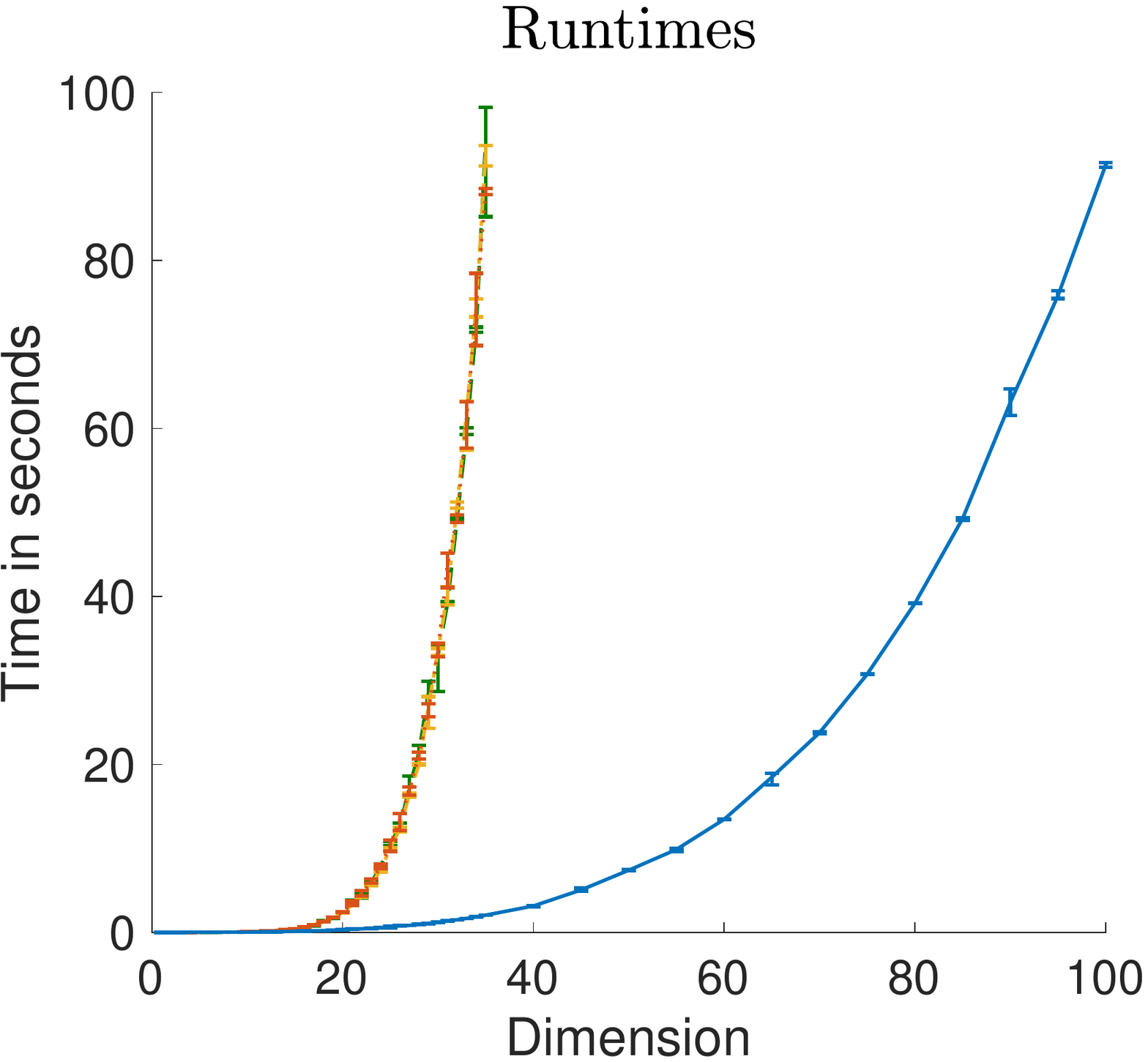}
  \vspace{-2.85cm}
  \caption{Runtimes for $n=3$, $1\leq m\leq 100$.} \label{time2}
 \end{minipage}
\end{figure}

 \begin{experiment}
We compare the computational runtimes of the approaches. 
To do so, we choose uniformly-distributed random function values $f_1,\dots,$ $f_{N(m,n)} \in [-1,1]$, $m,n \in \N$ as interpolation targets.
Then, we measure the time required to generate the unisolvent interpolation nodes $P_{m,n}$ and add the time taken 
to solve the PIP with respect to $f : \R^m \lo R$ with $f(p_i)=f_i$, $p_i \in P_{m,n}$, $i =1,\dots,N(m,n)$ by each approach.  
\end{experiment}

\begin{table}[t!]
 {\renewcommand{\arraystretch}{1.2}
\begin{center}
{\footnotesize
\begin{tabular}[t]{| l | l | l | l | l |} 
\hline 
Algorithm & Intervals & Degree  &  Pre-factor $p$ & Exponent $q$\\ 
 \hline
   \emph{Inversion} &$m =15,\dots,35 $ &  $n=3$&          $p =0.010737  $ &  $q=2.2982 $ \\
 \hline
 \emph{Linear Solver}   & $m =15,\dots,35 $ &  $n=3$&         $p=0.0076072  $ & $q=2.2907 $ \\ 
 \hline
 \emph{Linear Random Solver} &$m =15,\dots,35 $ & $n=3$ & $p=0.012009 $ &   $q=2.3289 $ \\ 
 \hline
  \emph{PIP-SOLVER} & $m=15,\dots,35$ & $n =3$&     $p=0.0076964 $ &   $q=1.2006  $ \\ 
 \hline
  \emph{PIP-SOLVER } & $m=15,\dots,100$ &  $n =3$&       $p=0.0034101  $ &   $q=1.2258 $ \\ 
 \hline
\end{tabular} 
}
\vspace{0.0cm}
\end{center}
\caption{Scaling of the computational cost by fitting the cost model $pN(m,n)^q$.}}\label{Tab}
\end{table}

The average and min-max span of the runtimes (over 10 repetitions with different $i.i.d.$~random function values; same 10 for every approach) are shown in Figures \ref{time1} and \ref{time2} versus the dimension $m$ for fixed degree $n=3$. 

While the actual problem size is $N(m,n)$, the dimension $m$ or the degree $n$ are more intuitive when characterizing a problem of fixed degree or fixed dimension, respectively.  
In low dimensions (inset figure) the \emph{Linear Random Solver} performs best due to its low overhead for generating unisolvent nodes. 
However, at about $m=9$ there is a  cross-over  above which the PIP-SOLVER is much faster than the other methods.
The absolute runtimes are below 0.05 seconds at the cross-over point, even though our prototype implementation of the PIP-SOLVER is not optimized.
As Figure \ref{time2} shows, even our simple implementation of the PIP-SOLVER can handle instances of dimension $m=100$ in the same time as the 
other methods require for $m=35$. Since $N(100,3)/N(35,3) \approx 20 $ the PIP-SOLVER outperforms the other approaches. 

The scaling of the computational cost with respect to problem size $N(m,n)$ is reported in Table \ref{Tab} for $n=3$. We fit all measurements with the cost model $pN(m,n)^q$. All fits show an \emph{R-square} of 1. 
We observe that the exponent of the cost scaling of the PIP-SOLVER is more than $1$ less than the exponents of the other methods with pre-factors that are never larger. The quadratic upper bound we have proven in this paper for the PIP-SOLVER holds in all tested cases. The other approaches roughly scale with an exponent of 2.3, as expected.

In addition to having a lower time complexity, the PIP-SOLVER also requires less memory than the other approaches. Indeed, the PIP-SOLVER requires only $\Oc(N(m,n))$ storage, 
whereas all other approaches require 
\linebreak$\Oc(N(m,n)^2)$ storage to hold the Vandermonde matrix. Due to this lower space complexity, we could solve the PIP for large instances, i.e., for $m >80$, where $N(m,3)\geq 10^5$ in less than 2 minutes, see Figure \ref{time2},  while classical approaches failed to solve such large instances  due to insufficient memory on the computer used for the experiment.

\begin{experiment} We measure the runtime of the PIP-SOLVER for different polynomial degrees $n$. Again, we choose uniformly-distributed random function values $f_1,\dots,f_{N(m,n)} \in [-1,1]$, $m,n \in \N$, as interpolation targets.
Then, we measure the time required to generate the multidimensional Newton-Chebyshev nodes $P_{m,n}$ and add the time taken 
to solve the PIP with respect to $f : \R^m \lo R$ with $f(p_i)=f_i$, $p_i \in P_{m,n}$, $i =1,\dots,N(m,n)$ for different $n$.   
\end{experiment}
Figure \ref{degree} shows the average and min-max span of the runtimes (over 10 repetitions with different $i.i.d.$~random function values) versus the dimension 
$m \in \N$ with logarithmic scale on the $y$-axis. We again fit the curves in the admissible intervals with the cost model $pN(m,n)^q$. Again, the goodness of fit as measured by the \emph{R-square} is 1 in all cases.  As expected, the exponent does not change much with degree. 
Just as in 1D, the almost linear scaling for low degrees reflects the computational power of the multivariate divided difference scheme. 
\begin{figure}[t!]
\begin{minipage}[t]{0.5\textwidth}
\vspace{-2.2cm}
  \includegraphics[scale=0.35]{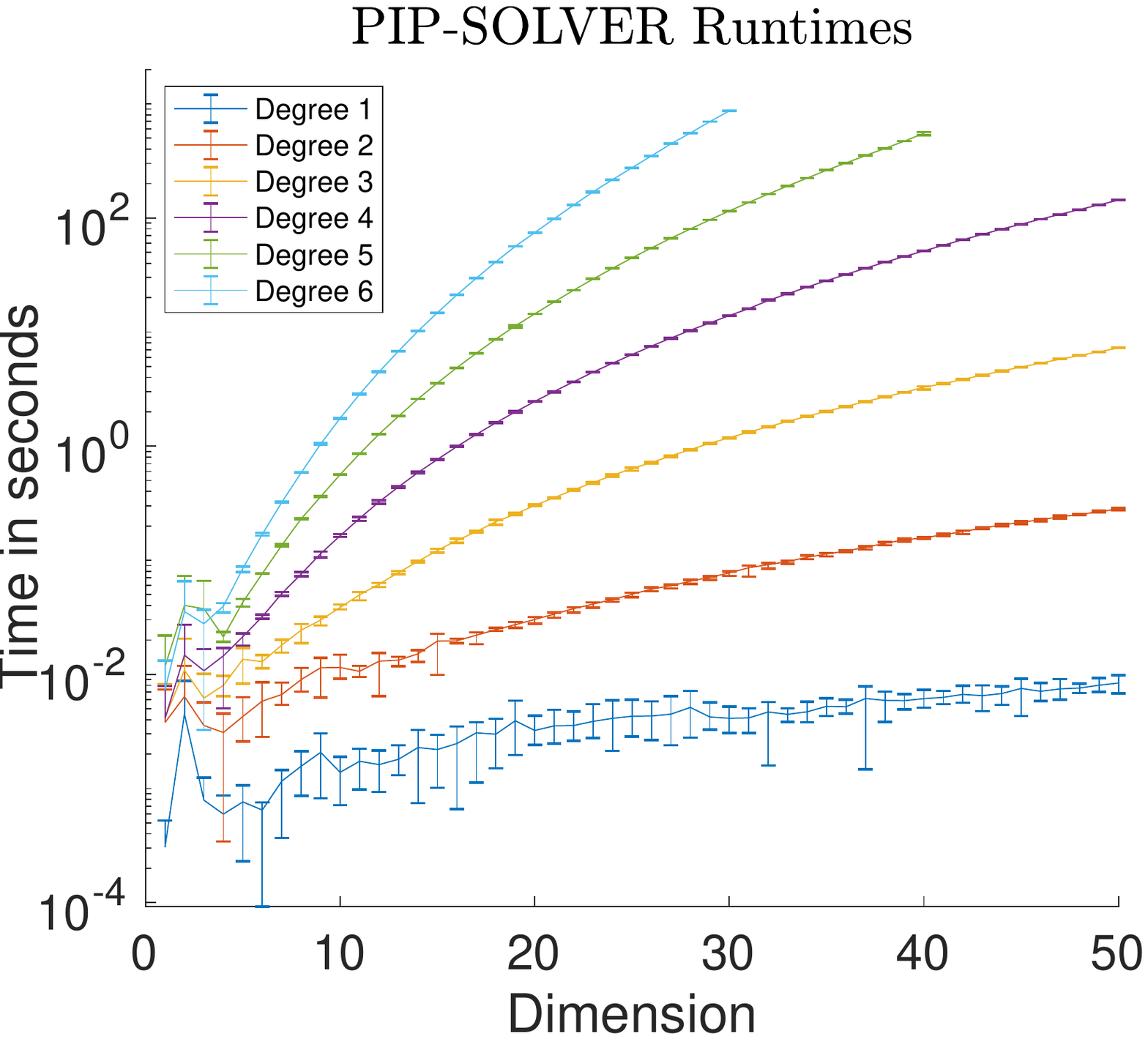}
  \vspace{-3.225cm}
 \end{minipage}
 \hspace{0.1cm}
 \begin{minipage}[t]{0.5\textwidth} 
 \vspace{0.5cm}
 {\footnotesize
\begin{tabular}[t]{ | l | l | l |} 
\hline 
  Degree  &  Pre-factor $p$ & Exponent $q$\\ 
 \hline

  $n=1$&          $p =0.0035219  $ &  $q=1.0450 $ \\
 \hline
  $n=2$&         $p=0.021732 $ & $q=1.1257$ \\ 
 \hline
 $n=3$ & $p=0.0031317$ &   $q=1.2096 $ \\ 
 \hline
$n =4$&     $p=0.0021351 $ &   $q=1.1861 $ \\ 
 \hline
 $n =5$&       $p=0.0017234  $ &   $q=1.1478 $ \\ 
  \hline
 $n =6$&       $p=0.0035746  $ &   $q=1.1336 $ \\ 
 \hline
\end{tabular} 
$ $\\
$ $\\
Fitting the cost model $pN(m,n)^q$
}
 \end{minipage}
  \caption{Runtimes of the PIP-SOVER for degrees $n=1,\dots,6$.}
  \label{degree}
\end{figure}

Next, we test the approximation properties of the PIP-SOLVER.
 A classic test case in approximation theory is \emph{Runge's function} 
  $$f_R(x)=\frac{1}{1+25x^2}\,.$$
  One can easily verify that $\left|\frac{\mathrm{d}^k f_R}{\mathrm{d}x^k}(1/5)\right| \lo \infty$, for $k \rightarrow \infty$.  Thus, 
  $f_R$ is smooth but unbounded with respect to  $\nol \cdot \nor_{C^\infty(\Omega)}$, which means that $f_R \not \in (C^\infty(\Omega),||~\cdot~||_{C^{\infty}(\Omega)})$.  
  This is the reason, why $f_R$ can not be approximated by interpolation with \emph{equidistant nodes} \cite{runge}. 
  However, as long as $k$ remains bounded, we have $\nol f_R \nor_{C^k(\Omega)} \leq C_K$ for some $C_K \in \R^+$. Thus, 
  for all $m \geq 1$, one verifies that the multivariate analog $f_R : \R^m \lo \R$ with $f_R(x)=  \frac{1}{1+25\nol x\nor ^2}$  satisfies  $f_R \in H^k(\Omega,\R)$ for $k >m/2$. By Theorem \ref{NCL}, all Sobolev functions can be approximated 
  when using Newton-Chebyshev nodes. 
  Therefore, we consider the multidimensional $f_R$ when testing the approximation abilities of the PIP-SOLVER. 
  
  \begin{experiment} We consider $m=5$, $\Omega=[-1,1]^m$ and use the PIP-SOLVER to compute the interpolant $Q_{m,n,f_R,CN}$ with respect to multidimensional \linebreak 
  Newton-Chebyshev nodes and the interpolant $Q_{m,n,f_R,EN}$
  with respect to multidimensional Newton nodes $P_{m,n}= \oplus_{i=1}^m\PP_i$ 
  with $\#\PP_i=n+1$ equidistant on $[-1,1]$. For technical resons we consider only even degrees $n \in 2 \N$ allowing to choose $0$ as the center of the Chebyshev nodes. 
  To estimate the $C^0$ distance between the interpolants, we   
   generate 400 uniformly random points $P \subseteq \Omega$ once and measure the relative approximation error  $|Q_{m,n,f_R,CN}(p) - f_R(p)|/f_R(p)$, $|Q_{m,n,f_R,EN}(p)-f_R(p)|/f_R(p)$ for each $p\in P$ and $n =2,4,\dots,24$. 
  \end{experiment}

  Figure \ref{Runge} plots the maximum and the mean of the relative distances \linebreak 
  $|Q_{m,n,f_R,CN}(p) - f_R(p)|/f_R(p)$, $|Q_{m,n,f_R,EN}(p)-f_R(p)|/f_R(p)$ over the 400 randomly chosen but fixed points, 
  with logarithmic scale in the $y$-axis. 
  Though the interpolant $Q_{m,n,f_R,EN}$ with equidistant nodes approximates $f_R$ for low degrees, it diverges with increasing degree $n \geq 6$. In contrast, the interpolant 
  $Q_{m,n,f_R,CN}$ continuous to converge to $f_R$ uniformly on $P$. Thus, we confirm that also in higher dimensions equidistant Newton nodes are infeasible for approximating  
  Runge's function, while Newton-Chebyshev nodes result in uniform convergence. 
  
\begin{figure}[t!]
\begin{minipage}[t]{1.0\textwidth}
\vspace{-3.5cm}
\center
  \includegraphics[scale=0.55]{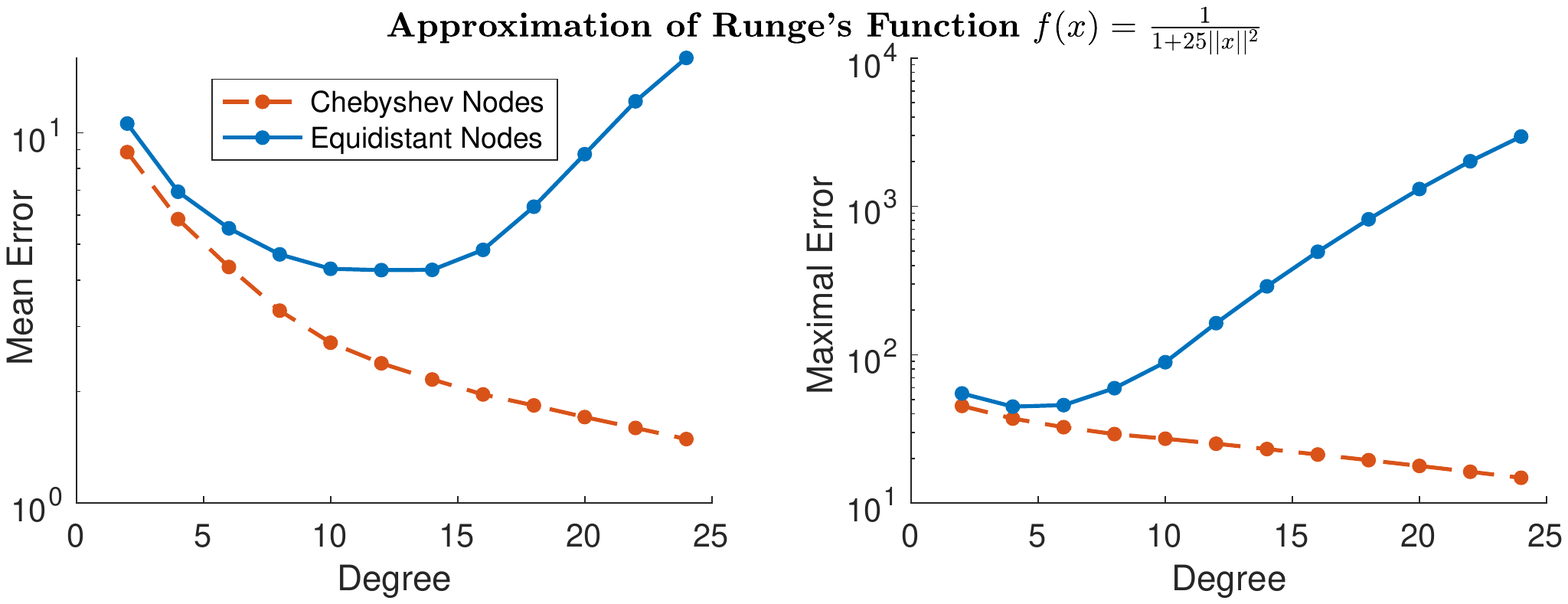}
  \vspace{-8.2cm}
 \end{minipage}
   \caption{Relative approximation error of Runge's function in fixed dimension $m=5$ with degrees $n=2,4,\dots,24$.}
   \label{Runge}
\end{figure}

\section{Potential Applications}\label{APP} We highlight several potential applications of the PIP-SOLVER in scientific computing and computational science. 
This list is by no means exhaustive, as PIPs are a fundamental component of many numerical methods. However, the following applications may not be obvious:  \\

{\bf A1)} \emph{Basic numerics:} Given $m,n\in \N$, $m\geq 1$, and a function $f : \R^m \lo \R$. It is classical in numerical analysis to determine the integral $\int_{\Omega}f \,\mathrm{d}\Omega$, $\Omega \subseteq \R^m$,   
and the partial derivatives  $\partial_{x_i}f(x) $, $1 \leq i \leq m$ of $f$. Upon the PIP, these desired quantities can easily be computed for the interpolation polynomial $Q_{m,n,f}$ of $f$. 
Due to our Main Result III (Theorem \ref{III}), $Q_{m,n,f} \xrightarrow[n\rightarrow\infty]{} f$ uniformly. Thus,  
$\int_{\Omega}Q_{m,n,f} \,\mathrm{d}\Omega \xrightarrow[n\rightarrow\infty]{} \int_{\Omega}f \,\mathrm{d}\Omega$ and by strengthen the conditions on $f$ 
also $\partial_{x_i}Q_{m,n,f}(x) \xrightarrow[n\rightarrow\infty]{} \partial_{x_i}f(x) $ uniformly.
A comparison with other approaches from numerical analysis is worth considering. \\

{\bf A2)} \emph{Gradient descent}  over multivariate functions is often used to (locally) solve (non-convex) optimization problems, where $\Omega \subseteq \mathbb{R}^m$ models the space of possible solutions for a given problem 
and $f: \R^m \lo \R$ is interpreted as an objective function. Thus, one wants to minimize $f$ on $\Omega$.   
Often, the function $f$ is not explicitly known $\forall x \in \Omega$, but can be evaluated point-wise.  Due to the Infeasibility of previous interpolation methods for \mbox{$m \gg 1$}, direct interpolation of $f$ on $\Omega$
was often not possible or not considered. Therefore, classical numerical methods like the \emph{Newton-Raphson iteration}
could not be used if analytical gradients were not available. Instead, \emph{discrete or stochastic gradient descent} methods were usually applied. However, these methods converge slowly and are potentially inaccurate. 
The PIP-SOLVER allows (locally) interpolating $f$ even for \mbox{$m \gg 1$} and enables (locally) applying classic \emph{Newton-Raphson methods}. Consequently, local minima could be found faster and more accurately. Moreover, the analytical representation of $Q$ potentially allows determining the global optimum if $f$ can be uniformly approximated \cite{PolyMin,Min}.\\

{\bf A3)} \emph{ODE \& PDE solvers:} 
A core application of numerical analysis is the approximation of the solution of Ordinary Differential Equations (ODE) or Partial Differential Equations (PDE). This always involves a (temporal and/or spatial) discretization scheme and a solver for the resulting equations. 
There are three classes of methods: collocation schemes, Galerkin schemes, and spectral methods. Spectral methods are based on Fourier transforms. Since FFTs are only efficient on regular Cartesian grids, 
spectral methods are hard to apply in complex geometries and on adaptive-resolution discretizations. The present PIP-SOLVER, however, is not limited to polynomial bases and could enable spectral methods on arbitrary distributions of 
discretization points in arbitrary geometries by interpolating with respect to a Fourier basis~\cite{IEEE}. Collocation methods can generally be understood as PIPs, as becomes obvious in the generalized formulation of 
finite-difference schemes and mesh-free collocation methods~\cite{Schrader}. The inversion of the Vandermonde matrix implied in mesh-free methods, or the choice of mesh nodes in compact finite-difference schemes, 
could benefit from the algorithms presented here. Finally, Galerkin 
schemes are based on expanding the solution of the differential equation in some basis functions, which is essentially what the PIP does for the orthogonal basis of monomials. 
Since the theory presented here is not restricted to this particular choice of basis, it is conceivable that similar algorithms can be formulated for other bases as well, potentially even for non-orthogonal ones. \\

{\bf A4)} \emph{Adaptive sampling} methods aim to explore a domain $\Omega \subseteq \R^m$ such that the essential information of $f : \Omega \lo \R$ is recovered. A classic example from statistics is multidimensional \emph{Bayesian inference} \cite{bayesian}, which relies on adaptive sampling methods. Mostly, these methods are based on 
\emph{Markov-Chain Monte-Carlo} or \emph{sequential Monte-Carlo} sampling. The notion of unisolvent node sets in high dimensions, as we have provided here, potentially helps design sampling proposals that explore $\Omega$ in a more controlled, complete, or more efficient way. \\

{\bf A5)} \emph{Spectral analysis:} The foundation of spectral analysis is to represent a function $ f : \Omega \subseteq \R^m \lo \R,\C$ with respect to some functional basis $f(x)= \sum_{i=1}^\infty c_ib_i(x)$, $c_i \in \R,\C$, $\spann\{b_i\}_{i\in\N}=L^2(\Omega,\R)$, e.g., with respect to Zernike or Chebyshev polynomials, linear (Fourier) harmonics, spherical harmonics, etc. This allows analyzing and understanding the essential 
character $f_0 =  \sum_{i\in I} c_ib_i(x)$ of $f$, where the finite set $I$ is chosen such that the $c_i$ with $i \in I$ cover the most relevant amplitudes. Interpolating $f$ with respect to the specified basis $b_i$ is the classic method of computing the coefficients $c_i$. In \cite{IEEE} 
we have already described how to extend the PIP-SOLVER to Fourier basis. An adaption to other functional bases can be done analogously, which potentially improves numerical spectral analysis.  \\

{\bf A6)} \emph{Cryptography:} A maybe surprising application is found in cryptography. There, the PIP is used to ``share a secret'' by choosing a random polynomial $Q \in \Z[x]$ with integer coefficients in dimension $m=1$. Knowing the values of $Q$ at $n+1$ different nodes (keys) 
enables one to determine \mbox{$Q(0) \!\!\! \mod p$} for some 
large prime number $p \in \N$. However, knowing only $n$ ``keys'' (nodes) prevents one for opening the ``door''.
Certainly, this method can be generalized to arbitrary dimensions using our approach \cite{Shamir}.
Since the PIP-SOLVER performs with machine accuracy, it can also prevent the reconstructed message from being corrupted by numerical noise. \\

There are many other computational schemes that require interpolation or are closely related to the PIP \emph{linear or polynomial regression} in machine learning.
We therefore close with a qualitative discussion of the PIP-SOLVER and state open questions, which potentially yield generalizations and further improvements in the future.

\section{Discussion and Conclusions}\label{Conc}

Even though \emph{Newton interpolation} in one dimension has been known since the 18$^\text{th}$ century, this may be the first generalization of this fundamental algorithm to arbitrary dimensions.
We have provided a complete characterization of the polynomial interpolation problem (PIP) in arbitrary dimensions and polynomial degrees. We have provided an algorithm called PIP-SOLVER
(see Algorithm \ref{alg:PIPSOLVER}) that computes the solution to generalized PIPs in 
$\Oc(N(m,n)^2)$ time and $\Oc(N(m,n))$ space, where $N(m,n)$ is the number of unknown coefficients of the interpolation polynomial with $m$ variables of degree $n$. 
We have shown that the algorithm generates unisolvent node sets for which the Vandermonde matrix is not only well conditioned, but has triangular form, 
enabling efficient and accurate numerical solution using a multivariate divided difference scheme also presented here. We have further provided the corresponding extensions of the Horner scheme, 
enabling evaluating the polynomial in $\Oc(N(m,n))$ and its integral and derivatives in $\Oc(nN(m,n))$ time. 
Lastly, we have studied the approximation properties of multivariate Newton interpolation polynomials and derived the notion of multidimensional Newton-Chebyshev nodes, 
showing that any Sobolev function of sufficient regularity can be uniformly approximated, and we have provided the corresponding upper bounds on the approximation errors. Taken together, 
these contributions solve Problem \ref{PolyInt} and answer Question \ref{Task} for arbitrary dimensions and degrees.

\begin{algorithm}
\caption{PIP-SOLVER}\label{alg:PIPSOLVER}
\begin{algorithmic}[1]
\Procedure{PIP-SOLVER}{$f,m,n$}\Comment{$f$ is computable}
\State Choose admissible plane $H$ according to Theorem \ref{GN} 
\If{$m=0$}
     \State \textbf{return} $f_{|H}$\Comment{$H$ is a point}
\EndIf
\If{$n=0$}
    \State \textbf{return} $f(0)$\Comment{$Q_{m,n,f}$ is a constant}
\EndIf
\State Choose $Q_H \in \Pi_{m,1}$ with $Q_H(H)=0$ 
\State $Q_1 = \text{PIP-SOLVER}(f_{|H},m-1,n)$\Comment{recursion over $m$}
\State $\widehat f = (f - Q_1)/Q_H $ on $\R^m\setminus H$ 
\State $Q_{m,n,f} = Q_1 + Q_H\cdot \text{PIP-SOLVER}(\widehat f,m,n-1)$\Comment{recursion over $n$}
\State \textbf{return} $Q_{m,n,f}$
\EndProcedure
\end{algorithmic}
\end{algorithm}

The problem statement and questions considered here are not new. Nor is the idea of decomposing the PIP w.r.t.~$m,n \in \N$ into sub-problems of dimension and degree \mbox{$(m-1,n)$} and \mbox{$(m,n-1)$}, respectively, which has already been mentioned in \cite{Guenther,2000}. In \cite{Gasca}, the cases $m=2,3$ were treated explicitly, while a generalization to arbitrary dimension was sketched and some characterizations of unisolvent nodes in arbitrary dimensions were given.   
However, the problem of computing the interpolation polynomial $Q_{m,n,f}$ efficiently and accurately remained unsolved. Indeed, all previous decomposition approaches were limited to relatively low 
dimensions and to nodes on pre-defined grids or meshes \cite{Bos,Erb,FAST,Gasca2000,Chung}, not providing a general algorithm for solving the PIP for arbitrary $m,n\in \N$.

Here, we were able to provide such a general algorithm and prove bounds on its time and space complexity. The achievement that made this possible was the introduction of unisolvent multidimensional Newton nodes $P_{m,n}$, which, combined with the multivariate Newton basis of $\Pi_{m,n}$, yielded a form of  the Vandermonde matrix $V_{m,n}(P_{m,n})$ that allows solving the system of linear equations $V_{m,n}(P_{m,n})C=F$  in less time than what is required for general matrix inversion. 

We demonstrated and validated our results in a practical software implementation of the PIP-SOLVER, showing scaling to high-dimensional spaces as common in applications ranging from machine learning to computational statistics. Owing to the elegance of the theory, the implementation of the PIP-SOLVER is straightforward and results in a simple code. 

Our simple reference implementation is not tuned for efficiency at the time of writing. In the future, we foresee distributed- and shared-memory parallel implementations in compiled programming languages to further reduce runtimes and enable even larger problems to be solved. This is possible since the recursive decomposition yields sub-problems that can be processed in parallel, using inter-process communication to ensure correct decomposition of the problem and synthesis of the final solution.

A possible extension of the presented theory is to also cover interpolation in other bases, such as the Fourier basis \cite{IEEE}, spherical harmonics, or Zernike polynomials. 
We also believe that our approach can be extended to multivariate barycentric or Lagrange Interpolation. In 1D, it is well known that Newton and Lagrange polynomials are related through the \emph{barycentric weights} 
\cite{berrut,werner}.
Precomputing these weights, barycentric Lagrange interpolation only requires linear time $\Oc(N(1,n))$ to compute the interpolant of degree $n\in \N$ in 1D. 
The present approach to multivariate Newton polynomials could lead to multivariate barycentric Lagrange interpolation schemes running in $\Oc(mN(m,n))$ for arbitrary $m,n \in \N$. 
Similarly, multivariate Hermite Interpolation could also be considered, since efficient realizations of this concept are closely related to Newton and Lagrange interpolation. In 1D, Hermite interpolation is a classic concept \cite{gautschi} that requires one to know the function $f : [-1,1]\lo \R$ and its derivatives on less than $n+1$ nodes in order to compute the interpolant $Q_{f,n}\in \Pi_{1,n}$. 

A case of special interest is spline interpolation \cite{unserSplines}. Fast implementations of spline interpolation are available \cite{fspline,unserFast}, making them a 
powerful and popular tool. Spline interpolation is based on decomposing the domain $\Omega =[-1,1]^m$ into smaller, shifted hypercubes $\Omega_i=[-\ee,\ee]^m + p_i$, $p_i \in \Omega$, $ i \in I$, $\ee >0$, 
and ``gluing'' the interpolants $Q_i : \Omega_i \lo \R$ to a $k$-times differentiable global function $Q$, $k \in \N$. Therefore, $\#I$ increases exponentially with dimension $m$. Tensorial formulations \cite{tensor,tensor2} are available for efficient local spline interpolation. However, the exponential scaling of the number of hypercubes, $\#I$ in which this has to be done cannot be overcome. This is why spline interpolation is mostly used for lower-dimensional problems. Further, the mathematical character of $f$ is not recovered in the spline basis and the approximation quality depends on the spline degree and on the choice of node conditions \cite{schoen,unserWiener,Unser:2005}. Therefore, spline interpolation is well suited  
to signal and image processing in low dimensions. The $L^2$-bases we proposed in $A5)$ above provide a potentially interesting choice in high dimensions. 
In principle, Hermite interpolation could also be used to glue spatially decomposed interpolants to a global $k$-times differentiable function. The notion of a globally unisolvent node set $P_{m,n}$ could then provide a way of spatially decomposing $\Omega$ such that the resulting global Hermite interpolant is of high approximation quality and can be computed efficient. 
We expect that this hybrid Hermite-spline interpolation method would relax some of the issues with splines in high dimensions.

The main practical limitation of our approach is that it requires the function $f : \Omega \lo \R$ to be computable in constant time, which means that the algorithm is free to choose the interpolation nodes. In many problems, however, $f$ is only known on a previously fixed node set $\PP \subseteq \R^m$ (i.e., the data given). In the case where $\PP$ is a (regular) grid, we can choose multivariate Newton nodes $P_{m,n}\subseteq \PP$ and our approach works. However, if $\PP$ is arbitrarily scattered, our approach does not directly apply. While resampling/reorganizing the data points can sometimes be an option, a general solution is outstanding. 
In particular, the optimal ordering of the nodes yielding the best numerical approximation in the sense of minimal rounding errors remains to be investigated. While this does not matter in infinite-precision arithmetic, finite-precision floating-point arithmetic accuracy, as well as algorithm speed, can be improved by appropriately ordering the points \cite{tal1988high}.

The main theoretical limitation of our approach is that we bounded the Lebesgue functions with respect to the $H^k$-norm for $k>m/2$. Additionally, we assumed that the considered functions $f \in C^0(\Omega,\R)$, $\Omega=[-1,1]^m$, $m \in \N$ are periodic.
While these assumptions match the requirements of many practical applications, the classic Lebesgue function estimates just require $f$ to be continuous. Hence, a deeper study of Lebesgue functions with respect to the powerful Sobolev analysis of periodic functions might improve the bounds presented here and might provide a way of controlling the convergence rate of  $Q_{m,n,f}\lo f$. 

Notwithstanding these open questions, we suspect that our concepts could provide a general perspective for considering multivariate interpolation problems, since, for dimension $m=1$, our concepts include the classic Newton interpolation scheme. We thus hope that the concepts and algorithms presented here will be useful to the community across application domains.


\bibliographystyle{spmpsci}      
 \bibliography{Ref.bib}   


%
%
%

\end{document}

%% file: PIP-tree.pdf_t
\begin{picture}(0,0)%
\includegraphics{PIP-tree.pdf}%
\end{picture}%
\setlength{\unitlength}{1243sp}%
\begingroup\makeatletter\ifx\SetFigFont\undefined%
\gdef\SetFigFont#1#2#3#4#5{%
  \reset@font\fontsize{#1}{#2pt}%
  \fontfamily{#3}\fontseries{#4}\fontshape{#5}%
  \selectfont}%
\fi\endgroup%
\begin{picture}(14245,7575)(-7349,59)
\put(4006,6194){\makebox(0,0)[lb]{\smash{{\SetFigFont{9}{10.8}{\rmdefault}{\mddefault}{\updefault}{\color[rgb]{0,0,0}$[\gamma]=(1,\dots,1,2,1)^T$}%
}}}}
\put(-1619,4169){\makebox(0,0)[lb]{\smash{{\SetFigFont{9}{10.8}{\rmdefault}{\mddefault}{\updefault}{\color[rgb]{0,0,0}$\gamma$}%
}}}}
\put(-1169,2819){\makebox(0,0)[lb]{\smash{{\SetFigFont{9}{10.8}{\rmdefault}{\mddefault}{\updefault}{\color[rgb]{0,0,0}$\mbox{m-1,n-l}$}%
}}}}
\put(2161,4889){\makebox(0,0)[lb]{\smash{{\SetFigFont{9}{10.8}{\rmdefault}{\mddefault}{\updefault}{\color[rgb]{0,0,0}$\mbox{m,n-1}$}%
}}}}
\put(1801,2819){\makebox(0,0)[lb]{\smash{{\SetFigFont{9}{10.8}{\rmdefault}{\mddefault}{\updefault}{\color[rgb]{0,0,0}$\mbox{m-1,n-1}$}%
}}}}
\put(6436,254){\makebox(0,0)[lb]{\smash{{\SetFigFont{9}{10.8}{\rmdefault}{\mddefault}{\updefault}{\color[rgb]{0,0,0}$\mbox{m,0}$}%
}}}}
\put(-7334,254){\makebox(0,0)[lb]{\smash{{\SetFigFont{9}{10.8}{\rmdefault}{\mddefault}{\updefault}{\color[rgb]{0,0,0}$\mbox{0,n}$}%
}}}}
\put(-3509,4979){\makebox(0,0)[lb]{\smash{{\SetFigFont{9}{10.8}{\rmdefault}{\mddefault}{\updefault}{\color[rgb]{0,0,0}$\mbox{m-1,n}$}%
}}}}
\put(-449,7139){\makebox(0,0)[lb]{\smash{{\SetFigFont{9}{10.8}{\rmdefault}{\mddefault}{\updefault}{\color[rgb]{0,0,0}$\mbox{m,n}$}%
}}}}
\put(-4274,7139){\makebox(0,0)[lb]{\smash{{\SetFigFont{12}{14.4}{\rmdefault}{\mddefault}{\updefault}{\color[rgb]{0,0,0}$T_{m,n}$}%
}}}}
\put(-4499,254){\makebox(0,0)[lb]{\smash{{\SetFigFont{9}{10.8}{\rmdefault}{\mddefault}{\updefault}{\color[rgb]{0,0,0}$\mbox{0,n-1}$}%
}}}}
\put(-1439,254){\makebox(0,0)[lb]{\smash{{\SetFigFont{9}{10.8}{\rmdefault}{\mddefault}{\updefault}{\color[rgb]{0,0,0}$\mbox{m-1,0}$}%
}}}}
\put(-3104,254){\makebox(0,0)[lb]{\smash{{\SetFigFont{9}{10.8}{\rmdefault}{\mddefault}{\updefault}{\color[rgb]{0,0,0}$\mbox{0,n-k}$}%
}}}}
\put(3421,254){\makebox(0,0)[lb]{\smash{{\SetFigFont{9}{10.8}{\rmdefault}{\mddefault}{\updefault}{\color[rgb]{0,0,0}$\mbox{m-1,0}$}%
}}}}
\put( 91,254){\makebox(0,0)[lb]{\smash{{\SetFigFont{9}{10.8}{\rmdefault}{\mddefault}{\updefault}{\color[rgb]{0,0,0}$\mbox{0,n-1}$}%
}}}}
\put(1936,209){\makebox(0,0)[lb]{\smash{{\SetFigFont{9}{10.8}{\rmdefault}{\mddefault}{\updefault}{\color[rgb]{0,0,0}$\mbox{0,n-k}$}%
}}}}
\end{picture}%

%% file: PIP-tree_D=1.pdf_t
\begin{picture}(0,0)%
\includegraphics{PIP-tree_D=1.pdf}%
\end{picture}%
\setlength{\unitlength}{1243sp}%
\begingroup\makeatletter\ifx\SetFigFont\undefined%
\gdef\SetFigFont#1#2#3#4#5{%
  \reset@font\fontsize{#1}{#2pt}%
  \fontfamily{#3}\fontseries{#4}\fontshape{#5}%
  \selectfont}%
\fi\endgroup%
\begin{picture}(11185,7530)(-4289,104)
\put(901,1694){\makebox(0,0)[lb]{\smash{{\SetFigFont{9}{10.8}{\rmdefault}{\mddefault}{\updefault}{\color[rgb]{0,0,0}$\mbox{0,k}$}%
}}}}
\put(-449,7139){\makebox(0,0)[lb]{\smash{{\SetFigFont{9}{10.8}{\rmdefault}{\mddefault}{\updefault}{\color[rgb]{0,0,0}$\mbox{1,n}$}%
}}}}
\put(-764,3089){\makebox(0,0)[lb]{\smash{{\SetFigFont{9}{10.8}{\rmdefault}{\mddefault}{\updefault}{\color[rgb]{0,0,0}$\mbox{0,n-l}$}%
}}}}
\put(-2429,4619){\makebox(0,0)[lb]{\smash{{\SetFigFont{9}{10.8}{\rmdefault}{\mddefault}{\updefault}{\color[rgb]{0,0,0}$\mbox{0,n}$}%
}}}}
\put(2971,4259){\makebox(0,0)[lb]{\smash{{\SetFigFont{9}{10.8}{\rmdefault}{\mddefault}{\updefault}{\color[rgb]{0,0,0}$\mbox{1,k}$}%
}}}}
\put(1621,5519){\makebox(0,0)[lb]{\smash{{\SetFigFont{9}{10.8}{\rmdefault}{\mddefault}{\updefault}{\color[rgb]{0,0,0}$\mbox{1,n-1}$}%
}}}}
\put(4591,2774){\makebox(0,0)[lb]{\smash{{\SetFigFont{9}{10.8}{\rmdefault}{\mddefault}{\updefault}{\color[rgb]{0,0,0}$\mbox{1,1}$}%
}}}}
\put(6616,1154){\makebox(0,0)[lb]{\smash{{\SetFigFont{9}{10.8}{\rmdefault}{\mddefault}{\updefault}{\color[rgb]{0,0,0}$\mbox{1,0}$}%
}}}}
\put(2431,254){\makebox(0,0)[lb]{\smash{{\SetFigFont{9}{10.8}{\rmdefault}{\mddefault}{\updefault}{\color[rgb]{0,0,0}$\mbox{0,1}$}%
}}}}
\put(-4274,7139){\makebox(0,0)[lb]{\smash{{\SetFigFont{12}{14.4}{\rmdefault}{\mddefault}{\updefault}{\color[rgb]{0,0,0}$T_{1,n}$}%
}}}}
\end{picture}%